\documentclass[preprint,12pt]{elsarticle}

\usepackage{enumerate}
\usepackage{mathrsfs}
\usepackage{amsmath,amssymb,color}
\DeclareMathOperator*{\esssup}{ess\,sup}
\usepackage{amsthm}

\usepackage{caption}

\usepackage{accents}
\providecommand{\undertilde}[1]{\underaccent{\tilde}{#1}}

\newtheorem{theorem}{Theorem}[section]
\newtheorem{remark}[theorem]{Remark}
\newtheorem{definition}[theorem]{Definition}
\newtheorem{lemma}[theorem]{Lemma}

\newtheorem{proposition}[theorem]{Proposition}

\usepackage[colorlinks=true,linkcolor=blue,citecolor=blue,urlcolor=blue]{hyperref}
\begin{document}

\begin{frontmatter}

\title{Boundary and Interior Control in a Diffusive Lotka-Volterra Model}

\author[a,c]{João Carlos Barreira}
\ead{jcbarreira95@gmail.com}

\author[b,c]{Maicon Sonego}
\ead{mcn.sonego@unifei.edu.br}

\author[c,d,e]{Enrique Zuazua}
\ead{enrique.zuazua@fau.de}


\address[a]{Instituto de Matemática e Estatística, Universidade Federal Fluminense, 
24210-201 Niterói, Rio de Janeiro, Brazil}

\address[b]{Instituto de Matemática e Computação, Universidade Federal de Itajubá, 
37500-903 Itajubá, Minas Gerais, Brazil}

\address[c]{Chair for Dynamics, Control, Machine Learning, and Numerics (Alexander von Humboldt-Professorship), 
Department of Mathematics, Friedrich-Alexander-Universität Erlangen-Nürnberg, 
91058 Erlangen, Germany}

\address[d]{Departamento de Matemáticas, Universidad Autónoma de Madrid, 
28049 Madrid, Spain}

\address[e]{Chair of Computational Mathematics, Fundación Deusto, 
48007 Bilbao, Basque Country, Spain}

\begin{abstract}

We investigate the controllability of a generalized diffusive Lotka-Volterra competition model for two species, incorporating boundary controls and an interior multiplicative control. Considering a smooth, bounded N-dimensional domain, we analyze ecologically pertinent scenarios characterized by constraints on both the controls and system states. Our results demonstrate how integrated control strategies can effectively overcome the limitations identified in previous studies. We prove two main results: (1) asymptotic controllability to single-species survival steady states under arbitrary system parameters, ensured by a combination of boundary and interior controls which act jointly to stabilize the system; and (2) finite-time controllability to a specific heterogeneous coexistence steady state via a two-phase strategy - first steering the system near the target with boundary control, then activating an interior multiplicative control in a localized region. The strong synergy between the two control mechanisms is crucial in both cases. We also analyze extinction dynamics and homogeneous coexistence, and support our findings with numerical simulations. The work concludes with perspectives for future research.\end{abstract}


\end{frontmatter}
\section{Introduction}


The Lotka-Volterra model (LVM) has found applications across a wide range of disciplines, including economics, technology, marketing, particularly in modeling industrial competition (see \cite{SOLOMON,WANG}) and even in language competition models \cite{KS,Patriarca}.

From a Control Theory perspective, LVMs offer a rich structure for exploration, as even minor modifications -- such as changes in the sign of certain coefficients -- can drastically affect the system's behavior and the nature of achievable control results. In this context, \cite{Bonnard} applied high-dimensional LVMs to control complex microbiotas through the use of probiotics, antibiotics, and combinations of transplants and bactericides. Optimal control techniques were employed to regulate microbial populations and reduce the incidence of infections caused by pathogens. In a related line of work, \cite{Ibañez} formulated an optimal control problem for an LVM capturing a predator-prey interaction, where the control variable represented hunting pressure on both species. The study focused on long-term behavior of optimal state-control trajectories, highlighting the ``turnpike" property -- i.e., the convergence of optimal trajectories toward a nearly steady regime when the time horizon is sufficiently long.

The literature on LVMs and their control is vast and diverse (see, for instance, \cite{DAGBOVIE,FISTER,HAO,PAVEL,YAN}), reflecting the model's relevance and flexibility. Moreover, the applicability of multiplicative controls to reaction-diffusion equations represented a key element in the design of this investigation; see, for instance, \cite{Khapalov2002, Khapalov2003, Lin, Lin2007}, which include results related to both asymptotic and finite-time controllability. In the recent works \cite{SONEGOZUAZUA} and \cite{afz}, the authors investigate the controllability of Lotka–Volterra systems using only boundary controls, considering weak and strong competition models, respectively.

Building upon these foundational results, our attention turns to their implications in real-world biological contexts. Such mathematical tools find natural application in ecological modeling scenarios, where control of population dynamics is of critical importance.

In this regard, the present work is motivated by the need to develop rigorous mathematical results for competitive LVMs, as such models are essential for understanding the ecological dynamics of interacting biological populations. These insights, in turn, are crucial for informed species conservation and ecosystem management. The version of the LVM analyzed in this article captures competition between two species and is particularly well-suited for examining how limited resources drive interspecies interactions. We undertake a detailed analysis of these dynamics, supported by both mathematical arguments and numerical simulations.

\subsection{Problem formulation and main results}\label{Subse_o_1_1}
In this article, we study a general diffusive LVM describing the dynamics of two competing species. The system involves two state variables, $u$ and $v$, representing the respective population densities of the species. It is equipped with boundary conditions and features a multiplicative internal control. In our first result, concerning asymptotic controllability towards the survival of only one of the species, the internal control must act on the entire domain, and thus the problem is formulated as follows
\begin{equation}\label{e1}
\left\{\begin{array}{ll}
u_t=d_1\Delta u+u(a_1-b_1u-c_1v)+hu,& (x,t)\in\Omega\times\mathbb{R}^+\\
v_t=d_2\Delta v+v(a_2-b_2u-c_2v),& (x,t)\in\Omega\times\mathbb{R}^+\\
(u(x,0),v(x,0))=(u_0,v_0),& x\in\Omega\\
u(x,t)=c_u(x,t),\ \ v(x,t)=c_v(x,t),& (x,t)\in\partial\Omega\times\mathbb{R}^+,
\end{array} \right.   
\end{equation}
where $a_i, b_i, c_i, d_i$ $( i = 1, 2) $ are positive parameters, and $ \Omega \subset \mathbb{R}^N$ ($ N = 1, 2, 3 $) is a smooth domain with a regular boundary $\partial\Omega$. 

The term $hu$ represents an interior multiplicative control, $c_u$, $c_v$ are boundary controls, and $u_0$ and $v_0$ represent the initial conditions of $u$ and $v$, respectively. For future reference, given $T > 0$, we denote by $Q$ the cylinder $\Omega \times (0, T)$, with lateral boundary $\Sigma = \partial\Omega \times (0, T)$.

For our second result, which is related to finite-time controllability towards a heterogeneous coexistence state, we consider a nonempty open set \( \omega \subset \Omega \) and modify the multiplicative control in \eqref{e1} to the form \( h 1_{\omega} u \), where \( 1_{\omega} \) denotes the characteristic function of \( \omega \). In this case, the problem is formulated as
\begin{equation}\label{e1_2}
\left\{\begin{array}{ll}
u_t=d_1\Delta u+u(a_1-b_1u-c_1v)+h1_{\omega}u,& (x,t)\in\Omega\times\mathbb{R}^+\\
v_t=d_2\Delta v+v(a_2-b_2u-c_2v),& (x,t)\in\Omega\times\mathbb{R}^+\\
(u(x,0),v(x,0))=(u_0,v_0),& x\in\Omega\\
u(x,t)=c_u(x,t),\ \ v(x,t)=c_v(x,t),& (x,t)\in\partial\Omega\times\mathbb{R}^+.
\end{array} \right.   
\end{equation}

Although the two systems are very similar, differing only in the way the control acts (globally in \eqref{e1} and locally in \eqref{e1_2}) it is convenient to present them separately. This distinction improves the organization of the exposition and contributes to a clearer understanding of the results.

This paper is devoted to analyzing the possibility of steering the system toward equilibrium configurations through the combined action of interior multiplicative and boundary controls. The proposed approach represents a distinctive contribution compared to previous works, which have primarily focused on controllability aspects for reaction-diffusion equations without exploring such a combined control strategy (see, e.g., \cite{afz,Kevin,BALETZUAZUA,SONEGOZUAZUA} and the references contained). As will be detailed below, the constraints considered in this work enhance the model’s relevance for practical applications by providing well-formalized biological interpretations. However, these constraints necessitate the use of regular controls, which in turn increase the analytical complexity of the problem. The adoption of an approach based on combined controls enabled us to obtain results on asymptotic controllability for two steady states -- both representing the survival of a single species -- as well as finite-time controllability for a steady state corresponding to a heterogeneous coexistence.

Note that, for the uncontrolled system, the carrying capacities of $u$ and $v$ are $a_1/b_1$ and $a_2/c_2$, respectively. Therefore, it is natural to impose constraints on the solutions based on these values, i.e.,
\begin{equation}\label{CH}
    0\leq u\leq a_1/b_1,\ \ 0\leq v\leq a_2/c_2.
\end{equation}

This necessitates imposing the same constraints on the boundary controls:
\begin{equation}\label{CUV}
    0\leq c_u\leq a_1/b_1,\ \ 0\leq c_v\leq a_2/c_2.
\end{equation}

Conditions \eqref{CH} and \eqref{CUV} introduce an additional level of difficulty to the problem,  making controllability properties harder to achieve; on the other hand, they reflect realistic physical constraints and represent one of the key features that make the problem both more challenging and more meaningful.

The existence of solutions corresponding to initial data and controls satisfying the bounds in \eqref{CH} and \eqref{CUV} is ensured by the classical theory of monotone methods (see, e.g., \cite{COS}), while regularity properties of the solution can be found in \cite{Amann,Henry}.

We begin by considering the steady state target $(0, a_2/c_2)$, which corresponds to the survival of one species ($v$) at its carrying capacity and the extinction of the other ($u$). In this case, the following theorem ensures the existence of boundary and interior controls -- depending only on the spatial variable -- that asymptotically steer the system's trajectories toward this target in infinite time.
\begin{theorem}\label{MT1}
  There are boundary controls $c_u, c_v\in L^{\infty}(\partial\Omega)$ and an interior control $h\in L^{\infty}(\Omega)$ such that, for every $(u_0,v_0)\in L^{\infty}(\Omega)\times L^{\infty}(\Omega)$, the solution $(u,v)$ of \eqref{e1}  satisfies
    $$\lim_{t\to\infty}(u(\cdot, t),v(\cdot,t))=(0,a_2/c_2)$$ 
    uniformly in $\Omega$.
\end{theorem}

Although constant boundary controls are the most natural way to drive a system to a steady state, their effectiveness strongly depends on the appropriate choice of parameters and the structure of the system. Theorem \ref{MT1} shows that the introduction of an internal multiplicative control of the form $(hu)$, acting on the entire $\Omega$, combined with suitable boundary controls, ensures asymptotic stabilization to the target state $(0, a_2/c_2)$, regardless of the parameter values.
This highlights the importance of internal multiplicative control: by influencing the internal dynamics of the system in a spatially distributed manner, it offers greater efficiency, especially in scenarios where constant boundary controls or additive controls are insufficient to achieve stabilization.

By symmetry, an analogous result to Theorem \ref{MT1} holds for the target $\left({a_1}/{b_1}, 0 \right)$, provided that the internal control (now denoted by $\bar{h}$) acts solely on the second equation of the system. In other words,
\begin{equation}\label{e2s}
\left\{\begin{array}{ll}
u_t=d_1\Delta u+u(a_1-b_1u-c_1v),& (x,t)\in\Omega\times\mathbb{R}^+\\
v_t=d_2\Delta v+v(a_2-b_2u-c_2v)+\bar{h}v,& (x,t)\in\Omega\times\mathbb{R}^+\\
(u(x,0),v(x,0))=(u_0,v_0),& x\in\Omega\\
u(x,t)=c_u(x,t),\ \ v(x,t)=c_v(x,t),& (x,t)\in\partial\Omega\times\mathbb{R}^+.
\end{array} \right.   
\end{equation}

\begin{remark}
Several remarks are in order.
\begin{itemize}    
 \item The demonstration proceeds according to the methodology described below.
\begin{enumerate}
 \item We begin by considering constant boundary controls given by 
\((c_u(x), c_v(x)) \equiv (0, a_2/{c_2})\), for all \(x \in \partial\Omega\), and the internal control chosen so as to stabilize the system. For this, we prove that $(0, a_2/c_2)$ is the unique steady state of the system corresponding to our control problem. 

\item The result follows from a convergence theorem established in \cite{JS}, which guarantees that the solution of the controlled system asymptotically approaches to a steady state.
\end{enumerate}

\item 	Observe that the controls employed are time-independent. Once applied, it is the system's intrinsic dynamics that gradually steer the state toward the desired target.

\item Note that the target is reached only asymptotically as $t \to \infty$. In the present setting, exact controllability in finite-time cannot be expected, since one component of the target steady state saturates the imposed constraints.

In general, achieving controllability in finite-time would require the system's trajectories to oscillate around the target, which would inherently lead to violations of the constraints. This type of obstruction is well known in the context of scalar diffusion models. Whether such an obstruction is also unavoidable for finite-time controllability in the current coupled system remains an interesting open question.

 \item  If the internal control were defined on an open subset $\widetilde{\omega} \subset \Omega$, assumptions regarding the parameters would be necessary; however, this was not the focus for this target. As previously described, in our next result, focusing on a different target, we consider internal control on an open subset of $\Omega$ and identify the assumptions regarding the parameters that will be required. 
\end{itemize} 
\end{remark}

As mentioned previously, internal multiplicative control plays a key role in stabilizing the system. More than that, in this context, this control mechanism has proven to be crucial in preventing the emergence of barrier functions, which could obstruct the system's trajectories and prevent it from reaching a desired target. As discussed in \cite{SONEGOZUAZUA}, when considering only boundary controls within a specific weak competition model -- namely, when in our system \eqref{e1} the parameters are set as $c_1, b_2 < 1, a_1 = b_1 = 1$, and $c_2 = 1$ -- a barrier effect can arise under certain conditions. Specifically, if $b_2 < a_2 < 1/c_1$ and $\lambda_1 < \min\{ (1 - a_2 c_1)/d_1,\ (a_2 - b_2)/d_2\} $, it is possible to construct barrier functions that prevent the system's trajectories from reaching the desired targets, regardless of the chosen boundary controls even when these vary in both space and time.

In such cases, internal control becomes essential. The presence of a multiplicative internal control, acting as a spatially distributed potential, allows the system to overcome the limitations imposed by these barriers and to steer gradually toward the target. For instance, when targeting the steady state $(0, a_2/c_2)$, the internal control $h$ directly affects the growth rate of the species $u$ across the entire domain. By sufficiently reducing this growth rate, the control can drive the population of $u$ toward extinction. Thus, the availability of internal multiplicative control eliminates the obstruction created by barrier functions and enables successful stabilization.

Naturally, a symmetric analysis can be carried out for the alternative target configuration.

Observe that, although our approach in this case requires the internal control to act over the entire domain, it guarantees asymptotic stabilization regardless of the parameter values. This is achieved through a multiplicative internal control, which can be interpreted as a particularly simple feedback control mechanism. Additionally, the proposed control strategy is independent of the initial conditions, further highlighting its robustness.

That said, it may be possible to localize the control action to smaller subregions of the domain by employing more sophisticated feedback control techniques (see, for example, \cite{BV} and the references therein), or by imposing additional structural assumptions on the problem, as discussed in \cite{A1,ACD}. This is an interesting open problem that deserves further consideration.

We now turn our attention to the second target, which represents a heterogeneous coexistence state in the specific scenario where $d_1 = d_2 = d$ and $a_1 = a_2 = a$,
\begin{equation}\label{uev}
(u^{**},v^{**})=(u^{**}(x),v^{**}(x))=\left(\left(\dfrac{c_2-c_1}{b_1c_2-c_1b_2}\right)\theta(x),\left(\dfrac{b_1-b_2}{b_1c_2-c_1b_2}\right)\theta(x)\right)
\end{equation}
where $\theta=\theta(x)$ is a smooth function that satisfies
\begin{equation}\label{ET}
\left\{\begin{array}{l}
d\Delta\theta+\theta(a- \theta)=0,\ \  x\in \Omega \\
\theta=0,\ \ x\in\partial\Omega\\
0<\theta<1,\ \  x\in \Omega.
\end{array}\right.
\end{equation}
In this case, we also assume \( b_1 > b_2 \), \( c_1 < c_2 \) to ensure the positivity of \( u^{**} \) and \( v^{**} \) in $\Omega$.

Before stating our next theorem, we need to consider the following eigenvalue problem
\begin{equation}\label{LA}
\left\{\begin{array}{ll}
-\Delta\phi= \lambda\phi,& x\in\Omega\\
\phi=0,& x\in\partial\Omega.
\end{array}\right.
\end{equation}

We denote by $\lambda_1$ the smallest eigenvalue of \eqref{LA}, which is well known to satisfy $\lambda_1 > 0$.
More details about the eigenvalue \( \lambda_1 \) will be discussed at the end of the Section \ref{Sec2}.

We are able to state our second result, which provides the main contribution of this article.
\begin{theorem}\label{MT3}
  Let any nonempty open set $\omega\subset\Omega$, if
     \begin{equation}\label{h12}
d<{a}/{\lambda_1}
     \end{equation}
    then for every $(u_0,v_0)\in L^{\infty}(\Omega)\times L^{\infty}(\Omega)$ there are boundary controls $c_u, c_v\in L^{\infty}(\partial\Omega)$ and an interior control $h\in L^{\infty}(\omega\times (0,T))$, such that the solution associated $(u,v)$ of \eqref{e1_2} satisfies
    \begin{equation}\label{exact_control}
        (u(\cdot, T),v(\cdot,T))=(u^{**},v^{**}),
     \end{equation}
    for some $T>0$.
\end{theorem}

Theorem \ref{MT3} states that the target $(u^{**},v^{**})$ can be exactly reached in finite-time when condition \eqref{h12} holds; its biological significance is examined in Section \ref{Sec4}. This inequality, which connects the geometry of the domain $\Omega$ (through the parameter $\lambda_1$), the diffusion capacity, and the reproduction rate of the species, is essential for the asymptotic controllability to the target -- representing the first stage in our control strategy. The second stage of our strategy involves applying finite-time controllability arguments for trajectories, using the stationary state $(u^{**}, v^{**})$ as the target trajectory. The system \eqref{e1_2} is reformulated based on this state, yielding a new nonlinear formulation in terms of the variables $(y, z) = (u - u^{**}, v - v^{**})$, and an associated linear system (see \eqref{linear_1}) with additive control of the form $\widetilde{h} 1_{\omega}$. Therefore, we will assume a nonempty open subset $\omega_{0} \Subset \omega$, since $\widetilde{h} 1_{\omega}=h1_{\omega}u^{**}$ and  $u^{**} > 0$ in $\overline{\omega_{0}}$. This assumption enables the application of Carleman estimates to the adjoint system associated with the linear equation, which in turn allows us to establish the global null controllability of the linear system. The details of this formulation are presented in Section \ref{Sec3}. Through the global null controllability of the linear system and estimates obtained for the state and control, we apply an invertibility theorem to establish the local null controllability of the nonlinear system of solution $(y, z)$, which in turn allows us to conclude finite-time local controllability for the system \eqref{e1_2} to the target $(u^{**}, v^{**})$. The combination of these two stages ultimately leads to the proof of Theorem \ref{MT3}.

In conclusion, Theorem \ref{MT3} combines an asymptotic approach with a local controllability argument around the steady state, resulting in a complete strategy to drive the system exactly to the target $(u^{**}, v^{**})$ in finite-time.

Note that it may be possible to reach the target in finite-time using only boundary controls. However, this is expected to be a challenging task due to constraints \eqref{CUV} and the fact that the target vanishes on the boundary.

\begin{remark}The following comments provide further insight into the result:
\begin{itemize}

\item The finite-time controllability result of Theorem \ref{MT3} is especially relevant because, although we have $0 < u^{**} < a_1/b_1$ and $0 < v^{**} < a_2/c_2$ for $x \in \Omega$, the target satisfies $u^{**} = v^{**} = 0$ over the boundary, and therefore saturates the constraints. This boundary behavior could, in principle, lead to constraint violations due to trajectory oscillations typically required to ensure finite-time controllability. However, our strategy successfully avoids this issue.

\item The key to this success lies in the use of a multiplicative interior control. Specifically, given any initial condition satisfying \eqref{CH}, we begin by setting $ h \equiv 0 $  and $ c_u \equiv c_v \equiv 0 $, allowing the natural system dynamics to steer the solution asymptotically toward the target. This occurs due to the stability properties of the system, in which hypothesis \eqref{h12} plays a fundamental role.  Once the state is sufficiently close to the target, we activate the interior control $ h $ and apply a finite-time local controllability argument to drive the system exactly to the target configuration (see Figure \ref{des}). 

\item Two features are fundamental to our control strategy:
\begin{enumerate}
    \item The solution components $ u $ and $ v $ remain zero on the boundary throughout the evolution, ensuring that the constraint \eqref{CUV} is never violated.
    \item The control $ h 1_\omega u $ is activated only when the trajectory is sufficiently close to the target, and thanks to the $ L^\infty $-bounds on $ h $, we can ensure it remains sufficiently small. This permits the use of comparison principles to verify the constraints imposed by \eqref{CH}.
    \end{enumerate}
    
\item Once again, the boundary controls \(c_u\) and \(c_v\) are time-independent; however, now the interior control $h$ is time-dependent.
\item As described in \eqref{e1_2}, the internal multiplicative control \( h1_{\omega} u\) acts on the first equation of the system; however, the implemented approach can be analogously adapted to the case where the control acts solely on the equation for \( v \), which would take the form \( \ddot{h}1_{\omega} v \).

\end{itemize}
\end{remark}

It is important to emphasize that both of the results described earlier rely on boundary controls as well as a multiplicative internal control. Finally, note that the targets related to the survival of a single species are much more delicate, since they require the internal control to act on the entire domain and, even so, we only achieve an asymptotic approximation to the target (Theorem \ref{MT1}). This difficulty is due to the assumed constraints. In contrast, for the target corresponding to the heterogeneous coexistence of species, although the states vanish on the boundary, we obtain controllability in finite-time using an internal control localized in the interior of the domain (Theorem \ref{MT3}).

\subsection{Biological interpretation and control}

Biologically, the interior multiplicative control $hu$ (or $h1_{\omega} u)$ represents an external agent that directly influences the reproduction rate of species $u$ within a region $\Omega$ (or subregion $\omega$) of the habitat. Through the diffusion term in the model, the effect of this local control naturally propagates throughout the entire domain. When $h > 0$, the control may correspond to a resource supplement that enhances the survival or reproductive success of $u$, or to genetic or environmental factors that promote its population growth. Conversely, when $h < 0$, the control reduces the growth rate of $u$, potentially modeling the impact of a pesticide, disease, or other harmful factor selectively affecting this species.

Importantly, in view of the biological interpretation adopted in this work, and in addition to the constraints imposed on the boundary controls and the states, it is essential to ensure that $h$ remains bounded. Classical control problems typically provide $L^2$-regularity for the control function, which is insufficient for the present setting. Therefore, the $L^{\infty}$ estimates for $h$ obtained in this study are fundamental. The techniques developed to achieve these bounds are sufficiently general and may be of independent interest beyond the context of this work.

The coupling structure of the system ensures that a control acting only on the first equation indirectly influences the second species. This relation can be succinctly expressed as:
\[
h \leadsto u \leadsto v.
\]

On the other hand, the boundary controls $c_u$ and $c_v$ act directly on the boundary population densities of $u$ and $v$, respectively. These controls admit various interpretations, such as the enforced migration of individuals at the habitat boundaries in a population management context, or the implementation of biological control measures like predator introduction or competitor removal at the edge of the habitat.

Although the analysis in this paper focuses on Dirichlet boundary controls, the methodology and conclusions naturally extend to the case of Neumann controls. In particular, once suitable Dirichlet controls guiding the system toward the desired state are identified, the associated boundary fluxes may be interpreted as Neumann controls. This duality is typical in problems involving full-boundary actuation and offers alternative interpretations relevant for practical scenarios such as flux regulation, filtration, or migration dynamics at domain boundaries.

\subsection{Outline}
In Section \ref{Sec2}, we prove Theorem \ref{MT1} and provide a detailed analysis to contextualize its significance. Section \ref{Sec3} is dedicated to the classical theory of controllability: we introduce the relevant Carleman inequalities, establish a null controllability result for a linearized system tailored to our framework, and present a local inversion theorem that will be instrumental in deriving Theorem \ref{MT3}.

In Section~\ref{Sec4}, we present the proof of Theorem~\ref{MT3}. As in Theorem \ref{MT1}, we also provide additional remarks aimed at clarifying the interpretation and broader implications of the result. 

In Section \ref{HCE}, we analyze the particular targets $(0,0)$ -- representing total extinction -- and
\begin{equation*}
(u^{*}, v^{*})=\left(\dfrac{a_1c_2-a_2c_1}{b_1c_2-b_2c_1},\dfrac{a_2b_1-a_1b_2}{b_1c_2-b_2c_1}\right),
\end{equation*}
which corresponds to a state of homogeneous coexistence (under conditions ensuring that $0 < u^{*} < a_1/b_1$ and $0 < v^{*} < a_2/c_2)$. Although we do not introduce new results in this section, the analysis of these targets is key to a deeper understanding of the overall control problem and helps to identify promising directions for future research. For example, driving the system to extinction $(0,0)$ by using internal control in only one equation appears to be a nontrivial task, regardless of the problem's coefficients. On the other hand, as shown in \cite{SONEGOZUAZUA}, the coexistence target $(u^{*}, v^{*})$ can be reached exactly in finite-time using boundary controls alone, without activating any internal controls.

Section \ref{NS} contains numerical simulations designed to illustrate and emphasize the impact of the multiplicative internal control. Finally, the concluding section summarizes the main contributions and outlines several open problems for further investigation.

\section{Proof of Theorem \ref{MT1}}\label{Sec2}
This section is dedicated to the proof of Theorem \ref{MT1}. In the interest of clarity and objectivity, we have chosen to present the proof in a simplified form. The more technical aspects and detailed justifications are addressed right after the proof, where the effects of activating the internal control in the \eqref{e1} system (and also in the \eqref{e2s} system) are discussed in detail. In particular, we demonstrate how the internal control enables trajectories to surmount potential barrier functions that may arise when relying exclusively on boundary controls, as evidenced in \cite{SONEGOZUAZUA}.

Before presenting the proof of Theorem \ref{MT1}, we state the following existence comparison theorem, which is an adaptation of Theorem 2.3 in \cite{pao}, and will be essential for understanding our result.
\begin{theorem}\label{cos}
Let $(\widetilde{u},\widetilde{v})$, $(\undertilde{u},\undertilde{v})$ be a pair of smooth functions such that $\widetilde{u}\geq\undertilde{u}\geq 0$ and $\widetilde{v}\geq\undertilde{v}\geq 0$. Moreover, suppose that $(\widetilde{u},\undertilde{v})$ satisfies 
\begin{equation}\label{SPEX}
\left\{\begin{array}{ll}
\widetilde{u}_t\geq d_1\widetilde{u}_{xx}+\widetilde{u}(a_1-b_1\widetilde{u}-c_1\undertilde{v}),& (x,t)\in (0,L)\times\mathbb{R}^+\\
\undertilde{v}_t\leq d_2\undertilde{v}_{xx}+\undertilde{v}(a_2-b_{2}\widetilde{u}-c_{2}\undertilde{v}),& (x,t)\in (0,L)\times\mathbb{R}^+\\
\widetilde{u}(x,0)\geq u_0(x),\ \ \undertilde{v}(x,0)\leq v_0(x),& x\in (0,L)\\
\widetilde{u}(x,t)\geq 0, \ \ \undertilde{v}(x,t)\leq 0,& (x,t)\in\{0,L\}\times\mathbb{R}^+,
\end{array}\right.
\end{equation}
and that $(\undertilde{u},\widetilde{v})$ satisfies
the corresponding reversed inequalities. Then the problem \eqref{e1} (with $h\equiv 0$) under zero Dirichlet boundary conditions has a
unique solution $(u, v)$ such that
$$\undertilde{u}(x,t)\leq u(x,t)\leq \widetilde{u}(x,t),\ \ \undertilde{v}(x,t)\leq v(x,t)\leq\widetilde{v}(x,t),$$
for $(x,t)\in [0,L]\times\mathbb{R}^+$.
\end{theorem}

\begin{proof}[Proof of Theorem \ref{MT1}]
 We detail how to achieve the asymptotic stabilization of the trajectory using the following controls 
\begin{equation}\label{C1}(c_u(x),c_v(x))\equiv(0,a_2/c_2)\,\,\mbox{ for $x\in\partial\Omega$, and } h(x)\equiv \sigma \,\,\mbox{for $x\in\Omega$}
\end{equation}
where $\sigma$ is a constant such that
\begin{equation}\label{sig1}
-a_1<\sigma<\lambda_1d_1-a_1.
    \end{equation}
    In fact, $h(x)$ can be any function satisfying \eqref{sig1}; here, we consider the constant $\sigma$ only for simplicity. 

Note that, if $d_1 > a_1/\lambda_1$ holds, then it is not necessary to activate the interior control to achieve asymptotic  stabilization. In this case, it suffices to consider $\sigma \equiv 0$ in \eqref{sig1}. 
Otherwise, we necessarily have $\sigma$ as a negative value and greater than $-a_1$. It follows that interference in the system through control modifies the carrying capacity of species $u$, which then becomes $a_1 + \sigma$.
Therefore, applying Theorem \ref{cos}, we obtain that the trajectories satisfy  
$$0 \leq u(x,t) \leq (a_1+\sigma)/b_1 < a_1/b_1, \quad 0 \leq v(x,t) \leq a_2/c_2.$$
  
Now, we will prove that $(0,a_2/c_2)$ is the unique steady state of 

\begin{equation}\label{e1p}
\left\{\begin{array}{ll}
u_t=d_1\Delta u+u(a_1-b_1u-c_1v)+\sigma u,& (x,t)\in\Omega\times\mathbb{R}^+\\
v_t=d_2\Delta v+v(a_2-b_2u-c_2v),& (x,t)\in\Omega\times\mathbb{R}^+\\
(u(x,0),v(x,0))=(u_0,v_0),& x\in\Omega\\
u(x,t)=0,\ \ v(x,t)=a_2/c_2,& (x,t)\in\partial\Omega\times\mathbb{R}^+.
\end{array} \right.   
\end{equation}

Let $(u^s,v^s)$ be a nonnegative solution of 
\begin{equation}\label{SPS}
\left\{\begin{array}{ll}
d_1\Delta u+u(a_1-b_1u-c_1v)+\sigma u=0,& x\in \Omega\\
d_2\Delta v+v(a_2-b_2v-c_2u)=0,& x\in \Omega\\
u(x,t)=0,\ \ v(x,t)=a_2/c_2,& x\in\partial\Omega.
\end{array}\right.
\end{equation}

By multiplying the first equation in \eqref{SPS} by eigenfunction $\phi(x)$ (which is positive in $\Omega$)  associated to the $\lambda_1$ (see \eqref{LA}), integrating over $\Omega$ and using Green's theorem twice together with the boundary condition on $u$, we obtain
$$-\displaystyle\int_{\Omega}d_1\Delta\phi u^sdx=\int_{\Omega}\phi u^s(a_1+\sigma-b_1u^s-c_1v^s)dx.$$
By \eqref{LA} and the non-negativity of $u^s$ and $v^s$,
$$\displaystyle\int_{\Omega}\phi u^s\left(d_1\lambda_1-a_1-\sigma\right)dx\leq -\int_{\Omega}b_1\phi(u^s)^2dx.$$
By \eqref{sig1}, we conclude that $u^s\equiv 0$. 

Now, we have the following problem
\begin{equation*}\label{SPSu}
\left\{\begin{array}{ll}
d_2\Delta v+v(a_2-c_2v)=0& x\in \Omega\\
v(x)=a_2/c_2,&x\in\partial\Omega,
\end{array}\right.
\end{equation*}
and we state that $v^s\equiv a_2/c_2$ is the unique solution. Indeed, if we take $w=a_2-c_2v^s$, then
\begin{equation*}\label{SPSw}
\left\{\begin{array}{ll}
-d_2\Delta w+w(a_2-w)/c_2=0& x\in \Omega\\
w(x)=0,& x\in\partial\Omega.
\end{array}\right.
\end{equation*}
By multiplying the equation by $w$ and integrating on $\Omega$,
we get
\begin{equation}\label{SER}
\int_{\Omega} d_2\vert\nabla w\vert^2dx=-\int_{\Omega} w^2(a_2-w)/c_2dx.
\end{equation}
Since $w(x)=0$ on $\partial\Omega$ and $w\leq a_2$, we conclude that  \eqref{SER} is true only if $w\equiv 0$. It follows that $v^s\equiv a_2/c_2$ and $(u^s,v^s)\equiv (0,a_2/c_2)$ is the unique stead state of \eqref{e1p}. By the main result of 
\cite{JS}, every bounded solution of \eqref{e1p} converge to a steady state, i.e.
\begin{equation*}\label{AB}
\displaystyle\lim_{t\to\infty}(u(x,t),v(x,t))=(0,a_2/c_2)
\end{equation*}
uniformly in $\Omega$, where $(u,v)$ is the solution of \eqref{e1} with controls given by \eqref{C1}.
 The theorem is proved.
\end{proof}

\begin{remark}Important remarks on Theorem \ref{MT1} are presented in the appropriate order.
\begin{itemize}
  
\item Theorem \ref{MT1} shows that when considering internal control, no assumptions about the parameters are needed for the targets to be asymptotically reached. Therefore, there are no barrier functions preventing trajectories from approaching the target $(0, a_2/c_2)$.

\item It is important to highlight that the action of the internal control can help trajectories overcome potential barrier functions that may arise when considering only boundary controls. These barrier functions prevent the trajectories from reaching their target, and in this sense, the internal control plays a fundamental role. For example, in \cite{SONEGOZUAZUA}, the authors considered a Lotka-Volterra problem under a weak competition regime, specifically: $a_1, b_1, c_2 = 1$, and $0 < c_1, b_2 < 1$. Taking into account only boundary controls, it was proven that if  $b_2<a<1/c_1$ and
$$\lambda_1<\min\left\{{(1-a_2c_1)}/{d_1},{(a_2-b_2)}/{d_2}\right\}$$  
then there exists  barriers functions that prevent certain trajectories from approaching the targets $(0, a_2)$ or $(1,0)$. In particular, for the target $(1,0)$, a barrier function is a non-trivial solution of the following problem
\begin{equation}\label{SPN}
\left\{\begin{array}{ll}
d_1\Delta u+u(1-u-c_1v)=0,& x\in \Omega,\\
d_2\Delta v+v(a_2-v-b_2u)=0,& x\in \Omega,\\
u(x)=1,\ \ v(x)=0,& x\in\partial\Omega.
\end{array}\right.
\end{equation}

\item In a biological context, we can interpret the inequality 
 \begin{equation}\label{ri}d_1 > a_1/\lambda_1\end{equation} as follows:
 \begin{enumerate}
     \item It is not necessary to activate the interior control, i.e., only boundary controls are needed for the asymptotic stabilization of the trajectories to the target $(0,a_2/c_2)$. This relation has an interesting geometric interpretation. It is well known that $\lambda_1$ continuously depends on $\Omega$  \cite{SMO,HA} and, in particular, when $\Omega$ is convex, we have that
\begin{equation}
\label{HA}
c(N)/\rho_{\Omega}^2\leq\lambda_1\leq C(N)/\rho_{\Omega}^2
\end{equation}
where $c(N)$, $C(N)$ are constants that depend only on the dimension $N$ and  $\rho_{\Omega}$ is the \textbf{inradius} of $\Omega$; that is, the radius of the largest ball contained in $\Omega$, 
$$\rho_{\Omega}:=\sup\{r>0; \exists\, x\in\Omega,\ \  B(x,r)\subset\Omega\}.$$
The inequality \eqref{HA} can be seen in \cite[Theorem 7.75]{HA}. Thus, the success of boundary controls ($c_u=0$ and $c_v=a_2/c_2$) will be achieved for the target $(0,a_2/c_2)$, when the inradius $\rho_{\Omega}$ is sufficiently small, meaning when the radius of the largest ball contained in $\Omega$ is small.
     
     \item For $c_u=0$ and $c_v=a_2/c_2$ to be effective in a determined domain $\Omega$, a high diffusion capacity of species $u$ will be required so that the role of the boundary control is diffused into the interior of the domain, affecting not only the population dynamics of species $u$ but also that of species $v$. 
     
     \item The control strategy will effectively drive the system towards the target $(0, a_2/c_2)$, provided that the intrinsic growth rate of $u$ is sufficiently small. Therefore, when condition \eqref{ri} is not satisfied, interior control must be activated. A careful reading of the theorem's proof reveals that this control directly influences the intrinsic growth rate of $u$, ensuring that \eqref{ri} holds with $a_1 + h(x)$ (i.e., $a_1 + \sigma$) instead of $a_1$.
 \end{enumerate}

Naturally, the preceding observations remain valid for the target $({a_1}/{b_1}, 0)$, assuming the internal control operates as described in equation \eqref{e2s}.
\end{itemize}
\end{remark}

\section{Preliminaries for the proof of Theorem \ref{MT3}}\label{Preliminares}\label{Sec3}
In this preliminary section, we aim to reformulate the system to achieve the target $(u^{**}, v^{**})$ through finite-time local controllability with multiplicative internal control. For simplicity, we will focus on the case where the control $h$ acts on the first equation of the system, as described in \eqref{e1_2}. The results established here will play a fundamental role in the proof of the locally  controllable in finite-time to the trajectory of Theorem \ref{MT3}, which represents the second stage in our strategy (Step 2), as detailed in Section \ref{Sec4}.

\subsection{Trajectory control and linear system}\label{subsection_linear_system}
To make the concepts more precise, we consider the system \eqref{e1_2} with Dirichlet boundary conditions equal to zero, that is,
\begin{equation}\label{e1_sem_controle_na_fronteira}
\left\{\begin{array}{ll}
u_t=d_1\Delta u+u(a_1-b_1u-c_1v)+h1_{\omega}u,& (x,t)\in\ Q\\
v_t=d_2\Delta v+v(a_2-b_2u-c_2v),& (x,t)\in Q\\
(u(x,0),v(x,0))=(u_0,v_0),& x\in\Omega\\
u(x,t)=0,\ \ v(x,t)=0,& (x,t)\in\Sigma.
\end{array} \right.   
\end{equation}
Let us consider a sufficiently regular trajectory $(\bar{u}, \bar{v}) = (\bar{u}(x,t),\bar{v}(x,t))$ solution to the following uncontrolled system
\begin{equation}\label{eq_trajectory}
\left\{\begin{array}{ll}
\bar{u}_t=d_1\Delta \bar{u}+\bar{u}(a_1-b_1\bar{u}-c_1\bar{v}),& (x,t)\in Q\\
\bar{v}_t=d_2\Delta \bar{v}+\bar{v}(a_2-b_2\bar{u}-c_2\bar{v}),& (x,t)\in Q\\
(\bar{u}(x,0),\bar{v}(x,0))=(\bar{u}_0,\bar{v}_0),& x\in\Omega\\
\bar{u}(x,t)= 0,\ \ \bar{v}(x,t)= 0,& (x,t)\in\Sigma,
\end{array} \right.   
\end{equation}
where $(\bar{u}_0,\bar{v}_0)\in L^{\infty}(\Omega)\times  L^{\infty}(\Omega)$.
\begin{definition}
    It will be said that \eqref{e1_sem_controle_na_fronteira} is locally  controllable in finite-time to the trajectory $(\bar{u},\bar{v})$ at time $T$ if there exists $\epsilon > 0$ with the following property: If $({u}_0,{v}_0)\in  L^{\infty}(\Omega)\times  L^{\infty}(\Omega)$ and $$\Vert {u}_0 - \bar{u}_{0}\Vert_{L^{\infty}(\Omega)}  + \Vert {v}_0 - \bar{v}_{0}\Vert_{L^{\infty}(\Omega)}\leq \epsilon,$$ 
    then there exist controls $h\in L^{\infty}(\omega\times (0,T))$ and associated states $(u,v)$ such that 
    \[
    u(x,T)=\bar{u}(x,T)\,\,\, \text{and}\,\,\, v(x,T)=\bar{v}(x,T)\,\, \text{in}\,\, \Omega.
    \]
\end{definition}

Note that, if we set \( u = y + \bar{u} \), \( v = z + \bar{v} \),  \( u_0 = y_0 + \bar{u}_0 \), and \( v_0 = z_0 + \bar{v}_0 \) we obtain by \eqref{e1_sem_controle_na_fronteira} and \eqref{eq_trajectory} the following nonlinear system
\begin{equation}\label{e3}
\left\{\begin{array}{ll}{y}_t=d_1\Delta {y}+{y}(a_1-b_1(2\bar{u} + y) - c_1z - c_1 \bar{v}) - c_1\bar{u} z\\
 + \,h1_{\omega}(y+\bar{u}),& (x,t)\in Q\\
{z}_t=d_2\Delta {z} + {z}(a_2-c_2(2\bar{v} + z)-b_2{y}- b_2 \bar{u}) - b_2\bar{v} y,& (x,t)\in Q\\
(y(x,0),z(x,0))=(y_0,z_0),& x\in\Omega\\
{y}(x,t)= 0,\ \ {z}(x,t)= 0,& (x,t)\in \Sigma.
\end{array} \right.   
\end{equation}

That said, we announce the following definition.
\begin{definition}
    Let any non-empty open set $\omega\subset\Omega$. It will be said that \eqref{e3} is locally null-controllable at time $T>0$ if there exists $\delta>0$ such that, for every $(y_0,z_0)\in L^{\infty}(\Omega)\times L^{\infty}(\Omega)$ with
    $$\Vert y_0\Vert_{L^{\infty}(\Omega)}  + \Vert z_0\Vert_{L^{\infty}(\Omega)}\leq \delta,$$ 
there exists controls $h\in L^{\infty}(\omega\times (0,T))$ and associated solutions $(y,z)$ satisfying
$$
y(x,T)=0\,\,\, \text{and}\,\,\, z(x,T)=0\,\, \text{in}\,\, \Omega.
$$
\end{definition}

Therefore, the local controllability in finite-time of the solution to $\eqref{e1_sem_controle_na_fronteira}$ for $(\bar{u},\bar{v})$ is equivalent to the local null controllability of the solution to \eqref{e3}.

\begin{remark}
    \begin{itemize}
        \item 
The use of finite-time controllability by trajectory is a well-established tool in control theory and is widely applied, see \cite{ANNA, Huaman, Zuazua2007}. 
\item Our goal is to employ this technique to ensure that the target $(u^{**}, v^{**})$ is reached exactly in finite-time. To do so, we introduce a subset $\omega_0 \Subset \omega$, i.e,  compactly contained in $\omega$. This assumption is motivated by the structure of the internal control in the corresponding linear system, which is given by $h 1_{\omega} u^{**}$. This formulation imposes the requirement that $u^{**}$ remain strictly positive, a property that is ensured within the closure $\overline{\omega}_0$. This point will be further clarified below.
 \end{itemize}
\end{remark}

It is important to highlight that we are considering the scenario of heterogeneous coexistence, where the conditions $a_1 = a_2 = a$, $d = d_1 = d_2$, $b_1 > b_2$, and $c_1 < c_2$ hold. Let $\omega_0$ be a non-empty open subset of $\mathbb{R}^{N}$ such that $ \omega_0 \Subset \omega$. We then consider the system in \eqref{e3} with $(\bar{u}, \bar{v}) = (u^{**}, v^{**})$, that is, our target is the steady state $(u^{**}, v^{**})$. Under this setting, the system can be reformulated as follows:
\begin{equation}\label{e4}
\left\{\begin{array}{ll}{y}_t=d\Delta {y}+{y}(a-b_1(2{u}^{**} + y) - c_1z - c_1{v}^{**}) - c_1{u}^{**} z\\
 + h1_{\omega}(y+{u}^{**}),& (x,t)\in Q\\
{z}_t=d\Delta {z} + {z}(a-c_2(2{v}^{**} + z)-b_2{y}- b_2 {u}^{**}) - b_2{v}^{**} y,& (x,t)\in Q\\
(y(x,0),z(x,0))=(y_0,z_0),& x\in\Omega\\
{y}(x,t)= 0,\ \ {z}(x,t)= 0,& (x,t)\in \Sigma,
\end{array} \right.   
\end{equation}
and its associated linear at zero is
\begin{equation}\label{linear_1}
\left\{\begin{array}{ll}{y}_t=d\Delta {y}+{y}(a-2b_1u^{**}- c_1 v^{**}) - c_1u^{**} z  + {\widetilde{h}}1_{\omega} + F_0,& (x,t)\in Q\\
{z}_t=d\Delta {z} + {z}(a - b_2 u^{**} - 2c_2v^{**} ) - b_2v^{**} y + F_1,& (x,t)\in Q\\
(y(x,0),z(x,0))=(y_0,z_0),& x\in\Omega\\
{y}(x,t)= 0,\ \ {z}(x,t)= 0,& (x,t)\in\Sigma,
\end{array} \right.   
\end{equation}
where we have considered
\begin{equation}\label{new_control}
    {\widetilde{h}}1_{\omega} = h1_{\omega}{u}^{**},
\end{equation}
and $F_0, F_1$ belonging to the appropriate weighted function spaces. From \eqref{new_control}, it is immediate that $\widetilde{h} 1_{\omega_0} = h 1_{\omega_0} u^{**}$, and it is important to emphasize that since $\omega_0 \Subset \omega$, $u^{**} > 0$ in $\overline{\omega}_0$.

In order to establish the local null controllability of the system \eqref{e4}, it is necessary first to demonstrate the global null controllability of the associated linear system \eqref{linear_1}. 

\subsection{Null controllability of linear system}

This subsection will be dedicated to verifying the global null controllability of the system \eqref{linear_1}. This will be covered briefly, as we will rely on the results presented by \cite{Clark}. 

We begin by introducing a new non-empty open set $\omega_1$, with $\omega_1 \Subset \omega_0$, so as to enable the application of a classical result by Fursikov and Imanuvilov \cite{Fur_Ima-96}.
\begin{lemma}
There exists a function $\eta\in C^{2}(\overline{\Omega})$ satisfying:
\begin{equation*}
    \left\{\begin{array}{cc}
    \eta(x)>0 & \forall x\in \Omega\\ \eta(x)=0& \forall x\in\partial\Omega\\
    \vert\nabla\eta(x)\vert > 0 & \forall x\in \overline{\Omega}\setminus \omega_{1}.
    \end{array}\right.
\end{equation*}
\end{lemma}
Then, let $m$ be a function that does not vanishing in $t=0$, i.e.,
\begin{equation*}
    m\in C^{\infty}([0,T]), \,\,\,\, m(t)\geq \frac{T^{2}}{8}\,\, \,\,\text{in}\,\, \,\,[0,T/2],\,\, \,\,m(t)=t(T-t)\,\, \,\,\text{in}\,\,\,\, [T/2,T].
\end{equation*}
Let us set 
\begin{equation}\label{gamma_e_beta}    \gamma(x,t):=\frac{e^{\lambda\eta(x)}}{m(t)},\,\,
    \beta(x,t):= \frac{e^{R\lambda}- e^{\lambda\eta(x)}}{m(t)},\,\,\text{where}\,\,R>\Vert\eta\Vert_{L^{\infty}(\Omega)}\,\, \text{and}\,\, \lambda>0.
\end{equation}
and let us introduce the notation
\begin{equation*}
    \begin{array}{lll}
        I(s, \lambda; \zeta) &:=& \displaystyle\iint_Q e^{-2s\beta}\left[ (s\gamma)^{-1}\left(|\zeta_{t}|^{2} + |\Delta\zeta|^{2}\right)\right.\\
        &\quad & \left. + \lambda^{2}(s\gamma)|\nabla\zeta|^{2} + \lambda^{4}(s\gamma)^{3}|\zeta|^{2}\right]dxdt.
    \end{array}
\end{equation*}

To announce Carleman inequality, we need the adjoint system of \eqref{linear_1} which is given by
\begin{equation}\label{adjoint}
    \left\{\begin{array}{ll}
         -\varphi_{t} = d\Delta \varphi + \varphi\left(a - 2b_1 u^{**} - c_1v^{**}\right)- b_{2}v^{**} \psi + H_0, & (x,t)\in Q\\
         -\psi_{t} = d\Delta\psi + \psi\left( a - b_2 u^{**} - 2c_{2}v^{**} \right) - c_1 u^{**}\varphi  + H_1,  & (x,t)\in Q\\
         (\varphi(x,T),\psi(x,T) )= (\varphi_{T} (x),\psi_{T}(x)),   & x\in \Omega\\
         \varphi(x,t) = 0,\,\,\,  \psi(x,t) = 0, & (x,t)\in\Sigma,
    \end{array}\right.
\end{equation}
where $(H_0, H_1)\in L^{2}(Q)\times L^{2}(Q)$ and $(\varphi_{t}(x), \psi_{T}(x))\in L^{\infty}(\Omega)\times L^{\infty}(\Omega)$. Therefore, by \cite{Clark}, we find that the following Carleman estimate is valid:
\begin{proposition}\label{Carleman_inicial}

Let $(H_0, H_1)\in L^{2}(Q)\times L^{2}(Q)$. There exist positive constants \( \bar{\lambda} \), \( \bar{s} \) such that, for any \( s \geq \bar{s}\) and \( \lambda \geq \bar{\lambda} \), there exists $\bar{C}(s,\lambda)$ with the following property: for any $(\varphi_T(x),\psi_T(x))\in L^{2}(\Omega)\times  L^{2}(\Omega)$, the associated solution to
\eqref{adjoint} satisfies
\begin{equation*}
\begin{array}{l}
I(s, \lambda; \varphi) + I(s, \lambda; \psi) \\
\leq \bar{C}(s,\lambda)\left[ \displaystyle\iint_Q e^{-2s\beta} \left(\gamma^{3}|H_0|^{2} + |H_1|^{2}\right)dxdt + \displaystyle\iint_{\omega_0\times(0,T)}e^{-2s\beta}\gamma^{7}|\varphi|^{2}dxdt\right].
\end{array}
\end{equation*}
Furthermore, $\bar{\lambda}$ and $\bar{s}$ only depend on $\Omega$, $\omega$, T,  $d$,
$(a - 2b_1 u^{**} - c_1v^{**})$, $ b_{2}v^{**}$, $( a - b_2 u^{**} - 2c_{2}v^{**})$ and $c_1 u^{**}$ and $\bar{C}(s,\lambda)$ only depend on these data $s$ and $\lambda$.
\end{proposition}
\begin{proof}
     Given that \( b_2 > 0 \) and \( v^{**} > 0 \) in \( \omega_0 \Subset \omega \), it follows that \( b_2 v^{**} \neq 0 \) in \( \overline{\omega}_0 \) from equation \eqref{adjoint}, and the result follows directly from Proposition 2.3 in \cite{Clark}.
\end{proof}

As a consequence of Proposition \ref{Carleman_inicial}, we obtain the null controllability of \eqref{linear_1} for ``small" right-hand sides \( F_0 \) and $F_1$  with weights that do not vanish when $t\to 0^{+}$.
\begin{proposition}\label{cnsl}
Assume that the functions \( F_0 \) and \( F_1 \) in \eqref{linear_1} satisfy
\[
\iint_Q e^{2s\beta}\gamma^{-3} \left(  |F_0|^2 + |F_1|^2 \right) dx dt < +\infty.
\]
Then \eqref{linear_1} is null-controllable, i.e., for any $(y_0, z_0)\in L^{2}(\Omega)\times L^{2}(\Omega)$, there exist controls $\widetilde{h}\in L^{\infty}(\omega_0\times (0,T))$ and associated states $(y,z)$ verifying 
\begin{equation}\label{regularidade_inicial_pesos_que_dependem_de_x_e_t}
    \begin{array}{c}
\displaystyle\iint_{Q}e^{2s{{\beta}}}\left({{\gamma}}^{-3}\vert y\vert^{2} + \vert z \vert^{2}\right)dxdt +  \displaystyle\iint_{\omega_0\times (0,T)}e^{2s{{\beta}}}{{\gamma}}^{-7}\vert \widetilde{h}\vert^{2}dxdt \\
+ \Vert\widetilde{h}\Vert^{2}_{L^{\infty}(\omega_0\times (0,T))}
\leq C\,
\kappa(y_0,z_0,F_0,F_1),
    \end{array}
\end{equation}
where
$$
\kappa(y_0,z_0,F_0,F_1) = \Vert y_0\Vert^{2}_{L^{2}(\Omega)} +  \Vert z_0\Vert^{2}_{L^{2}(\Omega)} + \displaystyle\iint_{Q}  e^{2s\beta}\gamma^{-3} \left(  |F_0|^2 + |F_1|^2 \right) dx dt.
$$
In particular, $y(x,T)=0$ and $z(x,T)=0$ in $\Omega$.  
In addition, it is also obtained that
\begin{equation}\label{regularidade_a_mais_com_pesos_que_dependem_de_x_e_t}
\begin{array}{lll}
\displaystyle\iint_{Q} e^{2s\beta}{\gamma}^{-5} (\vert \nabla y\vert^{2} + \vert \nabla z\vert^{2})dxdt\leq C\, \kappa(y_0,z_0,F_0,F_1).
\end{array}
\end{equation}
\end{proposition}
\begin{proof}
    See \ref{Appendix}.
\end{proof}

Now, to prove that the abstract map that rewrites the nonlinear problem \eqref{e4} (see \eqref{Mapa_A}) is well-defined, we need to obtain weighted estimates with weights that are independent of \( x \). To this end, we will define the following functions based on the weights defined in \eqref{gamma_e_beta}. Let us consider
\begin{equation*}
\begin{array}{cc}
   \beta^{\ast}(t)=\displaystyle\max_{x\in\overline{\Omega}}\beta(x,t),  &   \gamma^{\ast}(t)=\displaystyle\min_{x\in\overline{\Omega}}\gamma(x,t),\\
  \hat{\beta}(t)=\displaystyle\min_{x\in\overline{\Omega}}\beta(x,t),  &\hat{\gamma}(t)=\displaystyle\max_{x\in\overline{\Omega}}\gamma(x,t),
\end{array}
\end{equation*}
such that the constant $R$ is chosen large enough so that
\begin{equation}\label{desigualdade_dos_beta}
    2\hat{\beta} > \beta^{*}.
\end{equation}
Additionally, we introduce the notation
\begin{equation*}
\begin{array}{lll}
\widetilde{I}(s,\lambda;\phi) &:=& \displaystyle\iint_{Q}e^{-2s\beta^{*}}\left[(s{\hat{\gamma}})^{-1}(\vert\phi_{t}\vert^{2} + \vert\Delta\phi\vert^{2}) \right.\\
&\quad& \left. + \lambda^{2}(s\gamma^{*})\vert\nabla\phi\vert^{2} + \lambda^{4}(s\gamma^*)^{3}\vert\phi\vert^{2} \right]dxdt.
\end{array}
\end{equation*}

\begin{proposition}\label{Carleman_adoint}
    Let us assume that $(H_0, H_1)\in L^{2}(Q)\times L^{2}(Q)$. There exist positive constants $\lambda_0, s_0$ such that, for any $s\geq s_0$ and $\lambda\geq\lambda_0$, there exists $C_0(s,\lambda)$ with the following property: for any $(\varphi_T(x),\psi_T(x))\in L^{2}(\Omega)\times L^{2}(\Omega)$, the associated solution to \eqref{adjoint} satisfies
\begin{equation}\label{Carleman_1}
    \begin{array}{l}
       \widetilde{I}(s,\lambda;\varphi) + \widetilde{I}(s,\lambda;\psi) \\
       \leq C_{0}(s,\lambda)\left(\displaystyle\iint_{Q}e^{-2s{\hat{\beta}}}\left[{\hat{\gamma}}^{3}\vert H_0\vert^{2} + \vert H_1\vert^{2}\right] dxdt + \displaystyle\iint_{\omega_0\times (0,T)}e^{-2s{\hat{\beta}}}{\hat{\gamma}}^{7}\vert\varphi\vert^{2}dxdt\right).
    \end{array}
\end{equation}
Furthermore, $\lambda_0$ and $s_0$ depend only on $\Omega, \omega, T$ and the coefficients of $\varphi, \psi, \Delta\varphi, \Delta\psi$.
\end{proposition}
\begin{proof}
Follow the same reasoning used to obtain the Proposition \ref{Carleman_inicial}.
\end{proof}

The previous result, together with Proposition \ref{cnsl}, allows us to efficiently deduce the null controllability of system \eqref{linear_1}, with weights that are independent of the spatial variable.

\begin{proposition}\label{control_of_SL}
   Suppose the functions $F_0$ and $F_1$ in \eqref{linear_1} satisfy
    \begin{equation}\label{condi_ao_das_F}
       \begin{array}{l} \displaystyle\iint_{Q}e^{2s\beta^{*}}({\gamma^*})^{-3}\left(\vert F_{0}\vert^{2} + \vert F_1\vert^{2} \right)dxdt < + \infty.
       \end{array}
    \end{equation}
Then \eqref{linear_1} is null-controllable, i.e., for any $(y_0, z_0)\in L^{2}(\Omega)\times L^{2}(\Omega) $, there exist controls $\widetilde{h}\in L^{\infty}(\omega_0\times (0,T))$ and associated states $(y,z)$ verifying 
\begin{equation}\label{regularity_for_y_and_z}
    \begin{array}{c}
\displaystyle\iint_{Q}e^{2s{\hat{\beta}}}\left(\hat{{\gamma}}^{-3}\vert y\vert^{2} + \vert z \vert^{2}\right)dxdt +  \displaystyle\iint_{\omega_0\times (0,T)}e^{2s{\hat{\beta}}}{\hat{\gamma}}^{-7}\vert \widetilde{h}\vert^{2}dxdt\\
+ \Vert\widetilde{h}\Vert^{2}_{L^{\infty}(\omega_0\times (0,T))}
\leq C\,
\tilde{\kappa}(y_0,z_0,F_0,F_1),
    \end{array}
\end{equation}
where 
$$
\tilde{\kappa}(y_0,z_0,F_0,F_1) = \Vert y_0\Vert^{2}_{L^{2}(\Omega)} +  \Vert z_0\Vert^{2}_{L^{2}(\Omega)} + \displaystyle\iint_{Q}e^{2s\beta^{*}}({\gamma^*})^{-3}\left(\vert F_{0}\vert^{2} + \vert F_1\vert^{2} \right)dxdt.
$$
In particular, $y(x,T)=0$ and $z(x,T)=0$ in $\Omega$.  
In addition, it is also obtained that
\begin{equation}\label{regularidade_em_H1}
\begin{array}{c}
\displaystyle\sup_{[0,T]}\displaystyle\int_{\Omega}e^{2s{\hat{\beta}}}{\hat{\gamma}}^{-5} ( \vert  y\vert^{2} +  \vert  z\vert^{2}) dx +
\displaystyle\iint_{Q} e^{2s{\hat{\beta}}}{\hat{\gamma}}^{-5} (\vert \nabla y\vert^{2} + \vert \nabla z\vert^{2})dxdt \\
\leq  C\, \tilde{\kappa}(y_0,z_0,F_0,F_1).
\end{array}
\end{equation}
 \end{proposition}

\begin{remark}
A similar strategy may be employed when the internal multiplicative control, denoted as $\ddot{h}1_{\omega}v$, is applied exclusively to the second equation of \eqref{e1_2}. In this scenario, the internal control acts on the equation for \( v \) in \eqref{e1_sem_controle_na_fronteira}, and the procedure unfolds as follows:
  \begin{enumerate}
      \item In system \eqref{e3}, the control \( \ddot{h} 1_{\omega} (z + \bar{v}) \) would act on the second equation, resulting in the control \( \ddot{h} 1_{\omega} (z + v^{**}) \) on the \( z \)-equation of system \eqref{e4}.
    
    \item The control for the linear system takes the form \( \hat{h} 1_{\omega} := \ddot{h} 1_{\omega} v^{**} \), where \( v^{**} > 0 \) in \( \overline{\omega}_0 \).
    
    \item Since \( c_1 u^{**} \neq 0 \) in \( \overline{\omega}_0 \), an argument analogous to Proposition \ref{Carleman_inicial} holds.
    
    \item Consequently, we would also obtain a result analogous to Proposition \ref{cnsl}, ensuring the existence of a control \( \hat{h} \in L^{\infty}(\omega \times (0,T)) \) such that the associated solution \( (y,z) \) satisfies inequalities similar to \eqref{regularidade_inicial_pesos_que_dependem_de_x_e_t} and \eqref{regularidade_a_mais_com_pesos_que_dependem_de_x_e_t}.
    
    \item The subsequent propositions follow similarly, using the same reasoning.
  \end{enumerate}
\end{remark}

Once we have established the global null controllability of \eqref{linear_1}, we will define a Banach space that will encompass a reformulation of the null controllability problem for \eqref{e4}.
More specifically, we can reformulate the null controllability property for \eqref{e4} as abstract equations within appropriately chosen spaces of ``admissible" state controls. In particular, using the definitions applied to the equations and the admissible spaces we can show that such mappings are well-defined and \( C^1 \), and that their derivatives at the origin are surjective. This will allow us, by means of Lusternik's Inverse Mapping Theorem, to establish the local null controllability of the system under consideration. This invertibility result is a theorem in infinite-dimensional spaces that can be found, for example, in \cite{Alekseev}, and is stated below, where \( B_{r}(0) \) and \( B_{\delta}(\zeta_{0}) \) represent open balls with radii \( r \) and \( \delta \), respectively.

\begin{theorem}[Inverse Mapping Theorem]\label{Liusternik}
Let $ \mathcal{E}$ and $ \mathcal{Z}$ be Banach spaces and let $\mathcal{A}:B_{r}(0)\subset  \mathcal{E}\rightarrow  \mathcal{Z}$ be a $\mathcal{C}^{1}$ mapping. Let as assume that $\mathcal{A}^{\prime}(0)$ is onto and let us set $\mathcal{A}(0)=\zeta_{0}$. Then, there exist $\delta >0$, a mapping $W: B_{\delta}(\zeta_{0})\subset  \mathcal{Z}\rightarrow  \mathcal{E}$ and a constant $K>0$ such that
\begin{equation*}
\left\{\begin{array}{l}
    W(z)\in B_{r}(0)\ \ \text{and}\ \  \mathcal{A}(W(z))=z\,\, \, \, \forall\, z\in B_{\delta}(\zeta_{0}),\\
     \Vert W(z)\Vert_{ \mathcal{E}}\leq K\Vert z-\mathcal{A}(0)\Vert_{ \mathcal{Z}}\, \, \, \, \forall\, z\in B_{\delta}(\zeta_{0}).
    \end{array}\right.
\end{equation*}
In particular, $W$ is a local inverse-to-the-right of $\mathcal{A}$.
\end{theorem}

For more applications of these ideas in parabolic equations, see, for instance, \cite{CarvalhoLimaco}, \cite{CARALIMACOMENEZES} and \cite{CARALIMACOTHAMSTENMENEZES}.
  
\section{Proof of Theorem \ref{MT3}}\label{Sec4}
In this section, we prove Theorem \ref{MT3}. The argument is divided into two steps. First, under specific conditions on the parameters and considering the internal control identically null, we show that the asymptotic stabilization result for the target \( (u^{**}, v^{**}) \) is achieved. Second, we establish the local null controllability result using the Inverse Mapping Theorem (Theorem \ref{Liusternik}). The strategy is summarized in Figure \ref{des}.

A key point in the second step is that the target, which must be reached in finite-time, is strictly positive in \( \omega_0 \). Additionally, we make essential use of the regularity properties of the solutions to the linear system \eqref{linear_1} and the null control \( \widetilde{h} \) constructed in Proposition \ref{control_of_SL}. We conclude with a brief explanation of how these two steps together establish the theorem.

Due to the biological nature considered, we will provide, after the proof of the Theorem, a detailed analysis of how the result does not violate the constraints assumed in \eqref{CH} and \eqref{CUV}.

Before presenting the proof of the main theorem of this section, we state the well-known result regarding the existence of a positive solution to the logistic equation with diffusion (see \cite{PAY}).
\begin{lemma}\label{lemp}If $\alpha>\lambda_1$ and $\beta>0$, then  there exists a unique smooth function $z$ such that
\begin{equation}\label{EZP}
\left\{\begin{array}{l}
\Delta z(x)+z(x)(\alpha-\beta z(x))=0,\ \  x\in \Omega \\
z(x)=0,\ \ x\in\partial\Omega\\
z(x)>0,\ \  x\in \Omega.
\end{array}\right.
\end{equation}
\end{lemma}
\begin{proof}[Proof of Theorem \ref{MT3}]
{\bf Step 1: an asymptotic controllability result.} In this step our goal is to apply  some results of \cite{COS} related to the stability of a steady state solution of \eqref{e1_2} with $h\equiv 0$ and with zero Dirichlet conditions. So, here, under our assumptions, we consider the controls $(c_u,c_v)\equiv (0,0)$ and it suffices to prove the existence and uniqueness of the heterogeneous coexistence state $(u^{**},v^{**})$ for \eqref{e1_2}  under zero Dirichlet boundary conditions and $h\equiv 0$.

The existence is obtained due Lemma \ref{lemp} and our condition $d<a/\lambda_1$. Indeed, we obtain a function \( \theta \) that is a solution of \eqref{ET}, and a simple computation shows that \( (u^{**}, v^{**}) \) is a steady state of \eqref{e1_2}  under zero Dirichlet boundary conditions and $h\equiv 0$. The proof of uniqueness follows exactly the steps presented in Theorem 3.2 and Theorem 3.3 of \cite{COS} (see also Theorem 5 in \cite{SONEGOZUAZUA}, where a one-dimensional case  in a weak competition regime was addressed).

  {\bf Step 2: a local controllability result.} In order to show that the target $(u^{**},v^{**})$ is reached locally in finite-time, let us define a map $\mathcal{A}$ over control spaces of ``admissible'' states $\mathcal{E}$ and $\mathcal{Z}$. Let $\mathcal{E}$ be the space of functions
  \begin{equation*}
      \begin{array}{l}
\mathcal{E}:=\Big\{ (y,z,\widetilde{h}): \widetilde{h}\in L^{\infty}(\omega_0\times (0,T)),  \displaystyle\iint_{\omega_0\times (0,T)}e^{2s{\hat{\beta}}}{\hat{\gamma}}^{-7}\vert \widetilde{h}\vert^{2}dxdt< +\infty,\\
y, z, \partial_{i}y, \partial_{i}z, 

 y_t-d_1\Delta y, z_t-d_2\Delta z \in L^{2}(Q),\\
  \displaystyle\iint_{Q}e^{2s{\hat{\beta}}}\left({\hat{\gamma}}^{-3}\vert y\vert^{2} + \vert z \vert^{2}\right)dxdt< +\infty,\\

\text{for}\,\, F_0 =  {y}_t - d\Delta {y} - {y}(a-2b_1u^{**}- c_1 v^{**}) + c_1u^{**} z  - {\widetilde{h}}1_{\omega_0}\\

\text{and} \,\,\, F_{1} = {z}_t - d\Delta {z} - {z}(a - b_2 u^{**} - 2c_2v^{**} ) + b_2v^{**} y , \\

\displaystyle\iint_{Q}e^{2s\beta^{*}}(\gamma^{*})^{-3}\left(\vert F_{0}\vert^{2} + \vert F_1\vert^{2} \right)dxdt < + \infty,\,\, y(.,t)=z(.,t)=0\,\, \text{on}\,\, \Sigma,\\
 (y(.,0), z(.,0))\in L^{\infty}(\Omega)\times L^{\infty}(\Omega)\Big\},       
      \end{array}
  \end{equation*}
which is a Banach space for the norm $\Vert .\Vert_{ {\mathcal{E}}}$, with
\begin{equation*}
    \begin{array}{l}
    \Vert (y,z,\widetilde{h})\Vert^{2}_{\mathcal{E}} :=  \displaystyle\iint_{Q}e^{2s{\hat{\beta}}}\left({\hat{\gamma}}^{-3}\vert y\vert^{2} + \vert z \vert^{2}\right)dxdt + \displaystyle\iint_{\omega_0\times (0,T)}e^{2s{\hat{\beta}}}{\hat{\gamma}}^{-7}\vert \widetilde{h}\vert^{2}dxdt  \\
     + \displaystyle\iint_{Q}e^{2s{\beta^{*}}}(\gamma^{*})^{-3}\left(\vert F_{0}\vert^{2} + \vert F_1\vert^{2} \right)dxdt
   + \Vert y(.,0)\Vert^{2}_{L^{2}(\Omega)} + \Vert z(.,0)\Vert^{2}_{L^{2}(\Omega)}.
    \end{array}
\end{equation*}
Recall that $\widetilde{h} 1_{\omega_0} = h 1_{\omega_0} u^{**}$, and $u^{**} > 0$ in $\overline{\omega}_0$. Furthermore, note that an element $(y, z, \widetilde{h}) $ of $\mathcal{E}$ satisfies $y(x, T) = z(x, T) = 0$ in $\Omega$.

Also, let us introduce the Banach space $\mathcal{Z}:=  \mathcal{U}\times \mathcal{U}\times L^{2}(\Omega)\times L^{2}(\Omega)$, where
\begin{equation*}
    \begin{array}{l}
    \mathcal{U}:=\left\{ f\in L^{2}(Q) : \displaystyle\iint_{Q} e^{2s\beta^{*}}(\gamma^{*})^{-3}\vert f\vert^{2}dxdt < +\infty \right\}
    \end{array}
\end{equation*}
with the norm
$$   \Vert(f,\widetilde{f},g,\widetilde{g})\Vert_{\mathcal{Z}}^{2}:= \Vert f\Vert^{2}_{\mathcal{U}} +  \Vert \widetilde{f}\Vert^{2}_{\mathcal{U}} +  \Vert g\Vert^{2}_{{L^{2}(\Omega)}} +  \Vert \widetilde{g}\Vert^{2}_{{L^{2}(\Omega)}},$$ where $$\Vert f\Vert^{2}_{\mathcal{U}}:=\displaystyle\iint_{Q} e^{2s\beta^{*}}(\gamma^{*})^{-3}\vert f\vert^{2}dxdt.    
$$
Consider the map $\mathcal{A}:\mathcal{E}\rightarrow\mathcal{Z}$ such that
\begin{equation}\label{Mapa_A}
\begin{array}{c}
    \mathcal{A}(y,z,\widetilde{h}) 
    :=
    \Big( {y}_t - d\Delta {y} - {y}(a-b_1(2{u}^{**} + y) - c_1z - c_1 {v}^{**}) + c_1{u}^{**} z\\
    - \widetilde{h}1_{\omega_0}(\frac{y}{u^{**}}+1),
       {z}_t -d\Delta {z} - {z}(a-c_2(2{v}^{**} + z) - b_2{y}- b_2 {u}^{**}) + b_2{v}^{**} y,\\
        y(.,0),\, z(.,0) \Big).
    \end{array}
\end{equation}

Thus, by applying Proposition \ref{control_of_SL} and utilizing arguments analogous to those presented in \cite{Clark} (see Section 3), it can be demonstrated that the mapping $\mathcal{A}$, defined between the spaces $\mathcal{E}$ and $\mathcal{Z}$, satisfies the hypotheses of Theorem \ref{Liusternik}. A detailed proof is provided in  \ref{Appendix_B}.

Therefore, exists $\delta>0$ and a mapping $W:B_{\delta}(0)\subset {\mathcal{Z}}\rightarrow {\mathcal{E}}$ such that
\begin{equation*}
    W(z)\in B_{r}(0)\,\,\, \text{and}\,\,\, {\mathcal{A}}(W(z))=z, \,\,\, \forall z\in B_{\delta}(0).
\end{equation*}
In particular, taking $(0,0,y_{0},z_{0})\in B_{\delta}(0)$ and $(y,z,\widetilde{h})=W(0,0,y_{0},z_{0})\in {\mathcal{E}}$ we have
\begin{equation}\label{hfg}
    {\mathcal{A}}(y,z,\widetilde{h})=(0,0,y_{0},z_{0}),
\end{equation}
from which we infer that \eqref{e4} is locally null controllable at time $T > 0$ and, consequently, by the construction made in Subsection \ref{subsection_linear_system} we get the finite-time controllability on the target $(u^{**},v^{**})$. More precisely,
\[
u(x,T) = y(x,T) + u^{**} = u^{**}\,\,\, \text{and}\,\,\, v(x,T)=z(x,T) + v^{**} = v^{**}\,\,\, \text{in}\,\,\, \Omega.
\]
The {\it Step 2} is complete.

Finally, we are in a position to use the steps outlined above to complete the proof of the theorem. Consider an initial condition $(u_0,v_0)$ satisfying $0\leq u_0\leq a/b_1$, $0\leq v_0\leq a/c_2$. Given $\widetilde{T}>0$ and $\omega\subset\Omega$, by the {\it Step 2} there is $\epsilon> 0$ such that \eqref{e1_2} is locally controllable in finite-time towards $(u^{**},v^{**})$ assuming $c_u=c_v=0$ and some interior control $h_1\in L^{\infty}(\omega\times (0,\widetilde{T}))$. On the other hand, by the {\it Step 1} taking boundary controls $c_u=c_v=0$ and an interior control $h_2$ ($\equiv 0$) such that the respective solution $(u,v)$ of \eqref{e1_2} with initial condition $(u_0,v_0)$  satisfies
$$\Vert {u}(x,T_1) - u^{**}\Vert_{L^{\infty}(\Omega)}  + \Vert {v}(x,T_1) - v^{**}\Vert_{L^{\infty}(\Omega)}\leq \epsilon,$$ 
for some $T_1>0$. As a consequence, starting from time \( T_1 \), we apply the interior control \( h_1 \) over the interval \( (T_1, T_1 + \widetilde{T}) \), ensuring that the solution reaches exactly the target state $(u^{**}, v^{**})$ at the final time $T := T_1 + \widetilde{T}$, i.e,
$$u(x,T)=u^{**} \mbox{ and } v(x,T)=v^{**}.$$
In this case, the control strategy is defined as follows:
$$(c_u(x,t),c_v(x,t))=
(0,0),\quad (x,t)\in\partial\Omega\times (0,T)
$$
and
$$h(x,t)1_{\omega}=\left\{\begin{array}{ll}
0& (x,t)\in\Omega\times (0,T_1)\\
h_1(x,t)1_{\omega}& (x,t)\in\Omega\times (T_1,T).
\end{array}\right.$$
The theorem is proved. 
\end{proof}

\begin{figure}[h!]
        \centering
         \includegraphics[width=0.7\linewidth]{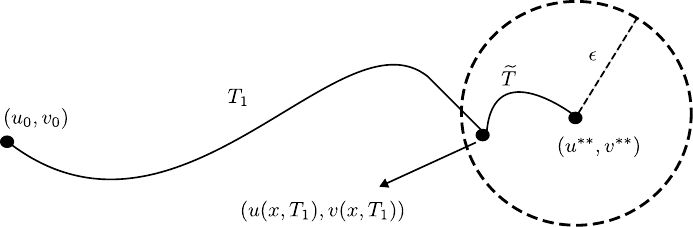}
   \captionsetup{font=small}
    \caption{Controlled trajectory reaching exactly the target $(u^{**},v^{**})$ at time $T=T_1+\widetilde{T}$.}
    \label{des}
\end{figure}

\begin{remark}Due to the symmetry in this case, the interior control action can be applied to either the first or the second equation.

Additional comments on Theorem \ref{MT3} are as follows.
    \begin{itemize}
        \item {\textbf{On the constraints.}}
\begin{enumerate}
    \item It is evident that \eqref{CUV} always holds since, in this case, we have $c_u\equiv c_v\equiv 0$;

\item The constraint \eqref{CH} requires a bit more attention. Indeed, since the boundary controls satisfy $c_u\equiv c_v\equiv 0$ and $u^{**}(x,t)=v^{**}(x,t)=0$ for $x\in\partial\Omega$, some trajectories could become negative near the boundary, given that we know that finite-time controllability requires oscillations of the trajectories around the target. However, this does not occur, and this is precisely where the estimate obtained in \eqref{regularity_for_y_and_z} plays a fundamental role.

More specifically, in the final arguments of the proof of Theorem \ref{MT3}, we can choose $T_1$ sufficiently large such
$$\Vert {u}(x,T_1) - u^{**}\Vert_{L^{\infty}(\Omega)}  + \Vert {v}(x,T_1) - v^{**}\Vert_{L^{\infty}(\Omega)}< a/\widetilde{C},$$ 
where $\widetilde{C}$ is the constant that appears in \eqref{regularidade_inicial_pesos_que_dependem_de_x_e_t}, which depends on $\widetilde{T}$ (see proof of Theorem \ref{MT3}), $\Omega$, $\omega$, and the coefficients of the system: $d$, $a$, $b_i$, $c_i$ ($i=1,2$). Now, by \eqref{hfg}, we observe that inequality \eqref{regularidade_inicial_pesos_que_dependem_de_x_e_t} for the control $h$  in \eqref{e1_2} holds with $F_0=F_1=0$, i.e.
  $$\Vert h\Vert_{L^{\infty}(\omega\times (0,T))}\leq\widetilde{C}(\Vert {u}(x,T_1) - u^{**}\Vert_{L^{\infty}(\Omega)}  + \Vert {v}(x,T_1) - v^{**}\Vert_{L^{\infty}(\Omega)})$$
and then
\begin{equation*}
\Vert h\Vert_{L^{\infty}(\omega\times (0,T))}< a.
\end{equation*}
Define $A=
\displaystyle\esssup_{(x,t) \in \Omega\times (0,T]} (a+h(x,t)1_{\omega})
$ (recall that $T=T_1+\widetilde{T}$) and note that $A> 0$. Now, if we consider
$$(\widetilde{u},\widetilde{v})=(A/b_1,a/c_2)\mbox{ and } (\undertilde{u},\undertilde{v})=(0,0)$$
we can use Theorem \ref{cos}  to conclude that the trajectories of \eqref{e1_2}  satisfy
$$0\leq u(x,t)\leq A/b_1,\ \ 0\leq v(x,t)\leq a/c_2$$
for all $(x,t)\in\Omega\times (0,T]$. Accordingly, it can be observed that:
\begin{enumerate}
\item The above inequalities ensure that the trajectories do not become negative. This conclusion is fundamental since the target $(u^{**},v^{**})$ vanishes at the boundary and, therefore, the oscillatory profile of the trajectories that  reach the target in finite-time could become negative at some point. The conclusions above show that this does not happen.

\item When $A > a$, the function $u(x,t)$ may exceed its upper bound. Biologically, this is not necessarily an issue, as interfering with the system through internal control naturally adjusts the species' carrying capacities. However, due to the way our control strategy was designed, this does not necessarily have to happen. We can choose $T_1$ sufficiently large so that $u(x,T_1)$ is close enough to $u^{**}$, making it possible to reach the target in finite-time with a very small $h$. Since $u^{**} < a$, using standard arguments on the continuous dependence of the system's coefficients, it is possible to prove that $u(x,t) < a$ for all $(x,t) \in \Omega \times (0,T]$.\\
\end{enumerate}
\end{enumerate}
\end{itemize}
\begin{itemize}  
\item {\textbf{Biological standpoint.}} 
\begin{enumerate}
\item Theorem \ref{MT3} addresses the controllability of a heterogeneous coexistence case \( (u^{**}, v^{**}) \), where the populations of both species \( u \) and \( v \) on the boundary are zero. In this context, we assume that the species shares the same diffusion capacity \( d \) and the same intrinsic growth rate \( a \). The strategy of maintaining \( c_u \equiv 0 \) and \( c_v \equiv 0 \) proves effective when the condition in \eqref{h12} is satisfied.

\item In contrast to Theorem \ref{MT1}, the characteristics of the target \( (u^{**}, v^{**}) \) and the proposed control strategy suggest one of the following conditions:
\begin{itemize}
    \item Low diffusion capacity of the species (i.e., \( d \) sufficiently small);
    \item High intrinsic growth rate (i.e., \textbf{a} sufficiently large);
    \item A domain \( \Omega \) with a large inradius (i.e., \( \lambda_1 \) sufficiently small).
\end{itemize}
\item The conditions presented in \text{(2)} have a natural biological interpretation: while sufficiently large  \textbf{a}  corresponds to a persistent search for a coexistence state, sufficiently small $d$ or $\lambda_1$ reduce the impact of null static control on the boundary.
\end{enumerate}
\end{itemize}
\end{remark}

\section{Homogeneous Coexistence and Extinction}\label{HCE}
We devote this section to a discussion of the steady states corresponding to homogeneous coexistence $(u^*, v^*)$ and extinction $(0,0)$.

\subsection{On the target \texorpdfstring{$(u^*,v^*)$}{(u^*,v^*)}}\label{sub_u_v}
In \cite{SONEGOZUAZUA}, the authors considered a Lotka-Volterra problem under a weak competition regime, specifically: $a_1, b_1, c_2 = 1$ and $0 < c_1, b_2 < 1$. Among other results, it was proven that the target corresponding to the homogeneous coexistence steady state, also denoted by $(u^*, v^*)$, can be reached in finite-time using only boundary controls. Under certain conditions on the parameters,
the strategy consists of approximating the target via Neumann controls and then applying a finite-time local  controllability result available in \cite{PZ}, for instance. This approach is feasible in this case because $u^*$ and $v^*$ are far from the values that constrain the controls and the state, unlike what happens with other targets. Indeed, the result of finite-time local controllability requires that the boundary controls oscillate locally above and below the target, thus potentially leaving their constraint intervals.

\subsection{On the target \texorpdfstring{$(0,0)$}{(0,0)}}\label{sub_0_0}
First, we must observe that the target $(0,0)$ is asymptotically controllable using only the constraints boundary controls $c_u$ and $c_v$, provided that the two inequalities below are satisfied

\begin{equation}\label{dd}d_1>{a_1}/{\lambda_1}\mbox{ and }d_2>{a_2}/{\lambda_1}.\end{equation}

Indeed, we can use arguments analogous to those employed in the proof of Theorem \ref{MT1}. This was done in \cite{SONEGOZUAZUA}, where a weak competition model was considered. If only one of the inequalities above is not satisfied, then it is possible to use an interior multiplicative control such that asymptotic stabilization towards $(0,0)$ occurs. For example, if $d_1 \leq a_1 / \lambda_1$, then we assume an internal control $h \equiv \sigma$ in \eqref{e1} such that  
$$-a_1 < \sigma < d_1 \lambda_1 - a_1$$
in order to proceed as in the proof of Theorem \ref{MT1}. On the other hand, if both inequalities in \eqref{dd} are not satisfied, then a multiplicative interior control in only one of the equations is not sufficient to steer the trajectories towards the target $(0,0)$. Obviously, this is possible by considering controls acting on both equations, which can also be studied but is not the aim of the present work.

\section{Numerical Simulations}\label{NS}

This section is dedicated to performing numerical simulations inspired by the theoretical results obtained above.

As we have seen, the dynamics of the problem considered here allow for the asymptotic stabilization of the target states $(0, a_2/c_2)$ and $(a_1/b_1, 0)$ through a combination of constrained boundary controls and a bounded interior multiplicative control.

We simulate the effect of the multiplicative control $h$ in driving the trajectories across barrier functions for the targets $(0,a_2/c_2)$ and $(a_1/b_1,0)$. For the target $(u^{**}, v^{**})$, the boundary controls guide the trajectories close to the target, while the internal multiplicative control $h$ ensures that the trajectory reaches the target exactly at time $T$. In this case, we simulate the optimal control with respect to the minimum time and compare the results with and without activating $h$.

In addition to visualizing the behavior of the trajectories and controls, our goal is also to observe that, as expected, the internal control $h$ can also contribute to the speed at which the trajectories can approaches the target.

The simulations were conducted using the software \textit{Matlab} and the package \textit{Casadi}.

For simplification and better understanding of the simulations, we chose a one-dimensional domain $\Omega=(0,L)\subset\mathbb{R}$. In this case, it is possible to explicitly determine the value of $\lambda_1$ in the estimate \eqref{h12}, namely, $\lambda_1=\pi^2/L^2$.

\subsection{Crossing barriers}
Here, we consider the following parameter values: $c_1 = 0.8$, $c_2 = 0.7$, $a_1 = a_2 = b_1 = b_2 = 1$, $d_1 = d_2 = 0.1$, and $L = 10$. In this case, it is possible to prove that a barrier exists for the target $(a_1/b_1,0)=(1,0)$ if no internal control is applied. The barrier is a nontrivial solution of \eqref{SPN} (with $a=1$), and its existence was proven in \cite{SONEGOZUAZUA}.  

Figure \ref{uvb} shows the barriers (black dashed lines) preventing the trajectories from approaching the target. Note that, in this case, we set the boundary values as $c_u=1$ and $c_v=0$, and the same would hold for any other values satisfying the assumed constraints.  

\begin{figure}[h!]
        \centering
         \includegraphics[width=1\linewidth]{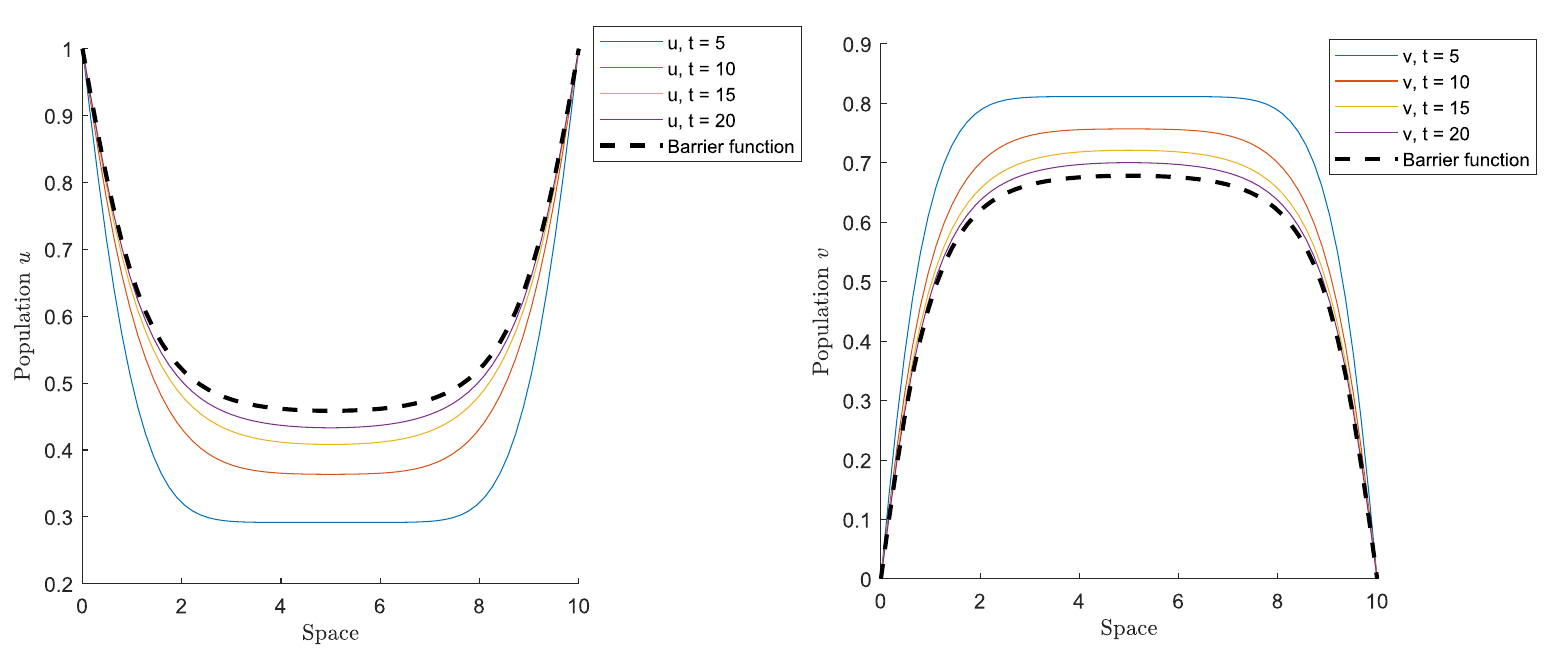}
   \captionsetup{font=small}
    \caption{Barriers preventing the trajectories from approaching the target $(a_1/b_1,0)=(1,0)$.}
    \label{uvb}
\end{figure}

By activating a multiplicative control in the interior of the domain, we can see in Figure \ref{uvbh} the trajectories crossing the barrier and approaching the target $(1,0)$. In this case, the control acted in the equation for $v$ (see \eqref{e2s}), and in this simulation, we set $\overline{h}=-0.8$.  

\begin{figure}[h!]
        \centering
         \includegraphics[width=1\linewidth]{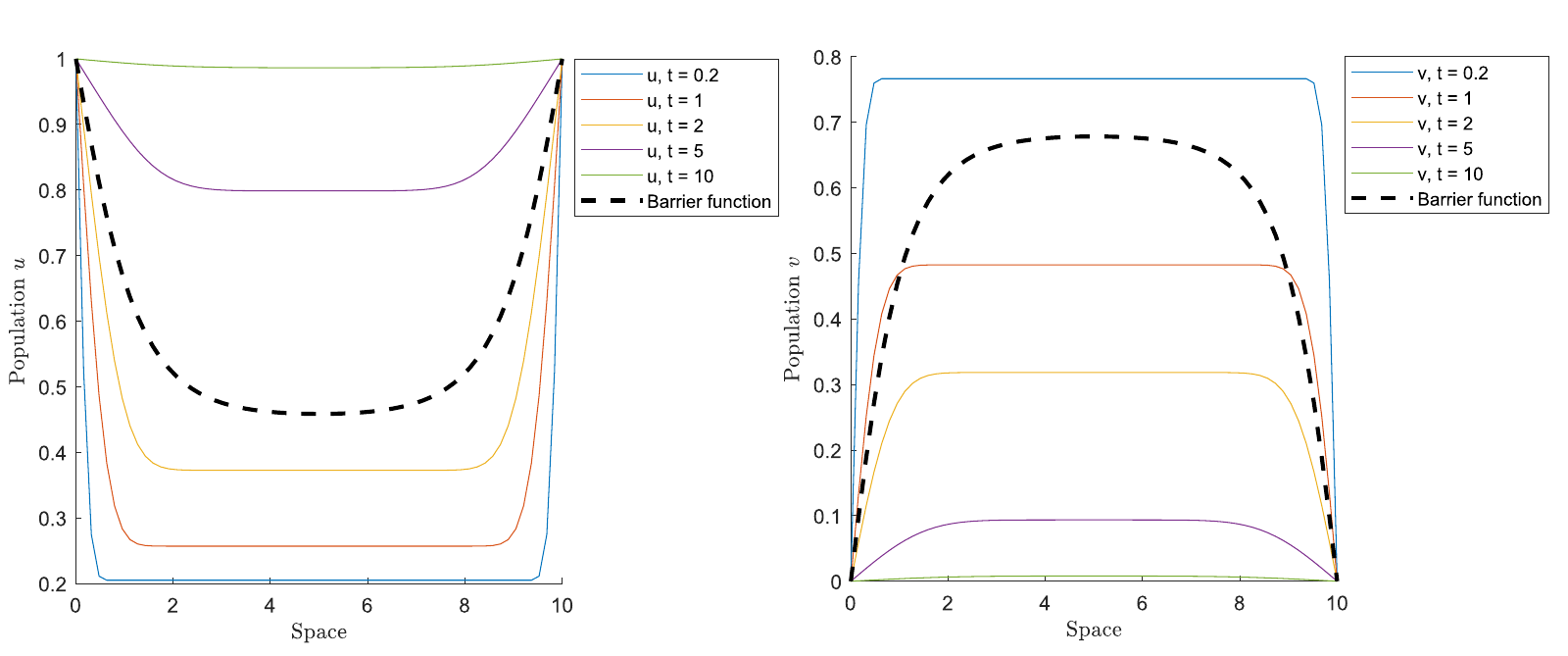}
   \captionsetup{font=small}
    \caption{Trajectories crossing the barrier under the action of the internal control $\overline{h}$.
}
    \label{uvbh}
\end{figure}

\subsection{Optimal control - minimum time}
In what follows, we present simulations involving the target $(u^{**},v^{**})$ (see \eqref{uev}). Here, we assume 
$d_1 = d_2=d=1$, 
$a_1 = a_2=a= 10$, $b_1 = 1.8$, $c_1 = 0.2$, 
 $b_2 = 1$,  $c_2 = 1.4$ and $L=1$. 
Moreover, we assume the  initial condition $(u_0,v_0)=(0.2,0.3)$ and, to simplify the simulation, we considered \( \omega = (0,L) = (0,1) \). Under these assumptions, inequality \eqref{h12} is satisfied as well as the other conditions of Theorem \ref{MT3}. With the parameters adopted above, the target $(u^{**},v^{**})$ can be observed in Figure \ref{uv2}.

\begin{figure}[h!]
        \centering
         \includegraphics[width=0.6\linewidth]{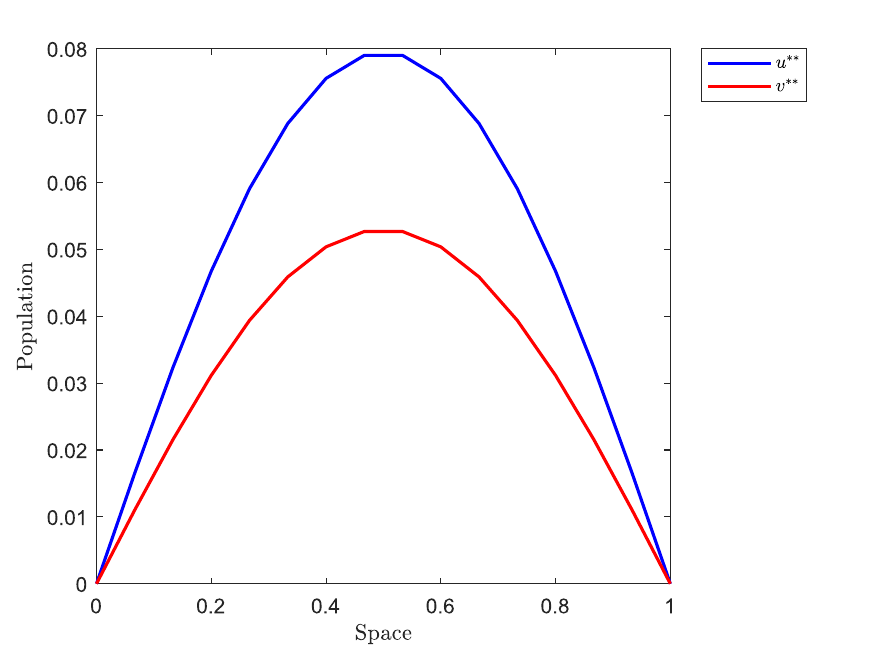}
   \captionsetup{font=small}
    \caption{The target $(u^{**},v^{**})$.}
    \label{uv2}
\end{figure}

We are interested in the controls $c_u$, $c_v$, and $h$ that drive the trajectory to reach the target (with a tolerance of $10^{-2}$) starting from the initial condition $(u_0, v_0) = (0.2, 0.3)$. We recall that constraints must be assumed on the controls $c_u$ and $c_v$ (see \eqref{CUV}), and $h$ must be bounded. In our example, we again assume $|h| \leq 1$. The minimum time obtained was $t = 1.8055$, and Figures \ref{uh4} and \ref{vh4} display the behaviors of $u$ and $v$, respectively.
\begin{figure}[h!]
        \centering
         \includegraphics[width=0.8\linewidth]{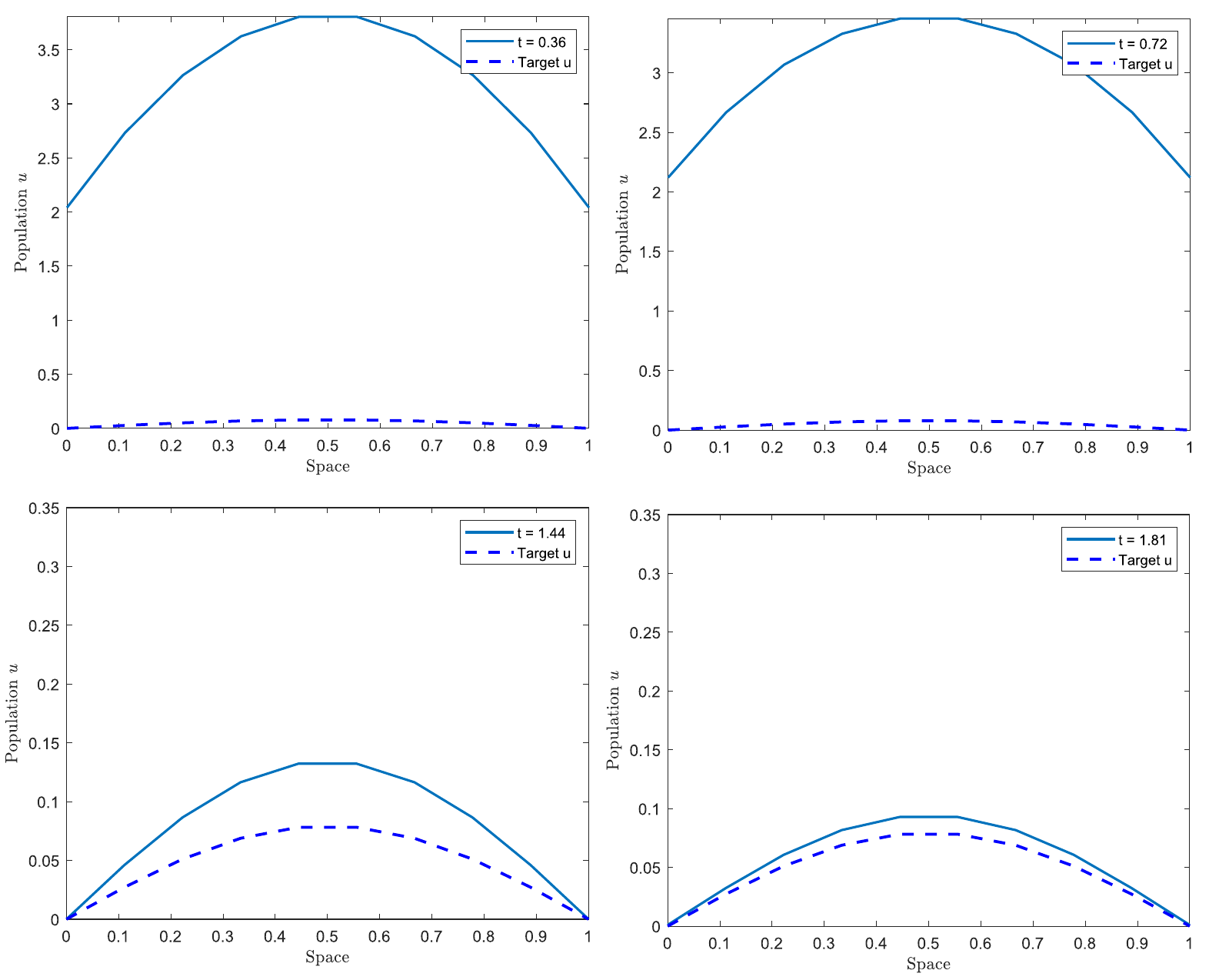}
   \captionsetup{font=small}
    \caption{Trajectories approaching $u^{**}.$
}
    \label{uh4}
\end{figure}

\begin{figure}[h!]
        \centering
         \includegraphics[width=0.8\linewidth]{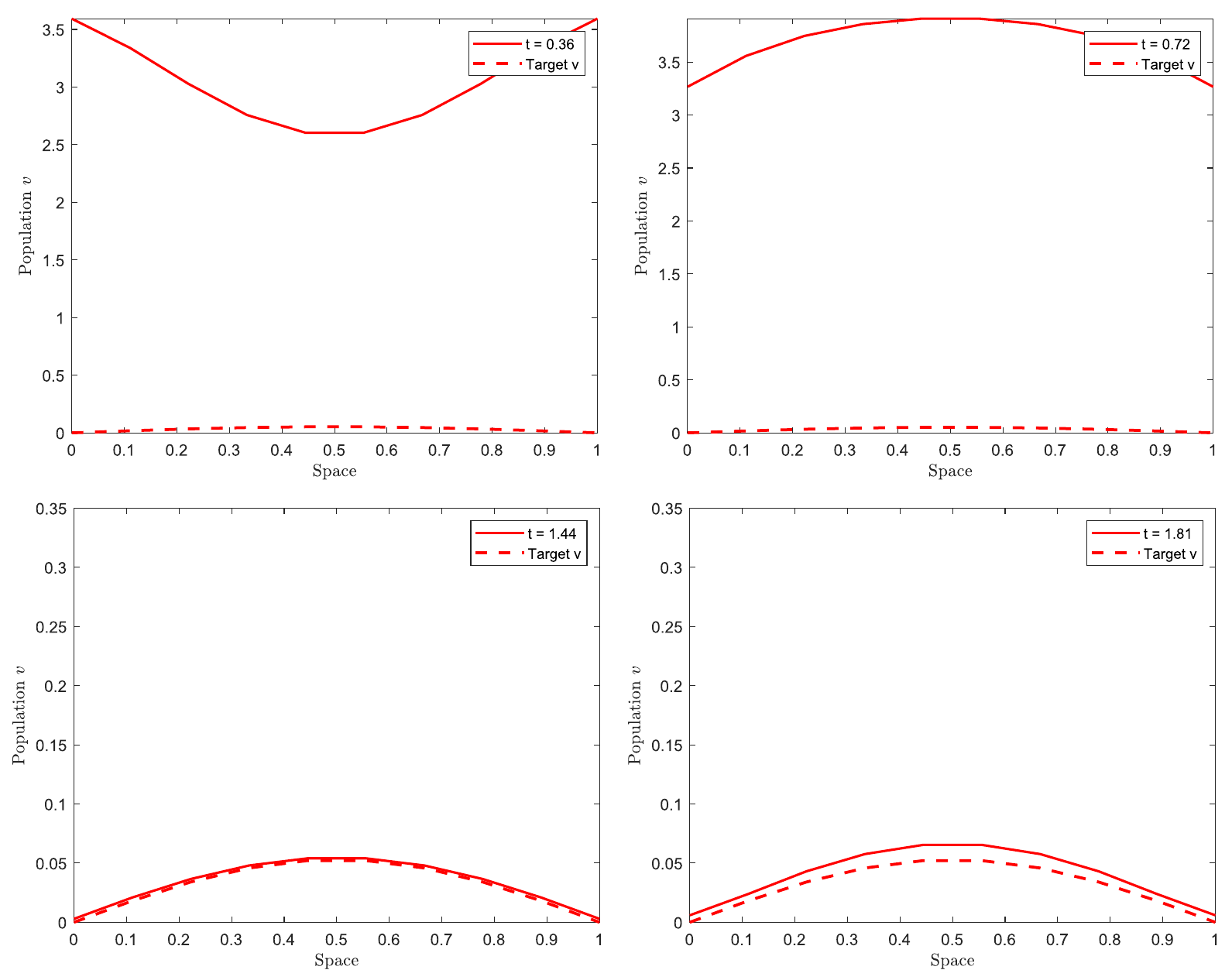}
   \captionsetup{font=small}
    \caption{Trajectories approaching $v^{**}$.}
    \label{vh4}
\end{figure}
The behavior of the controls $c_u$, $c_v$, and $h$ can be observed in Figures \ref{cucvh} and \ref{h}, respectively.

\begin{figure}[h!]
        \centering
         \includegraphics[width=0.8\linewidth]{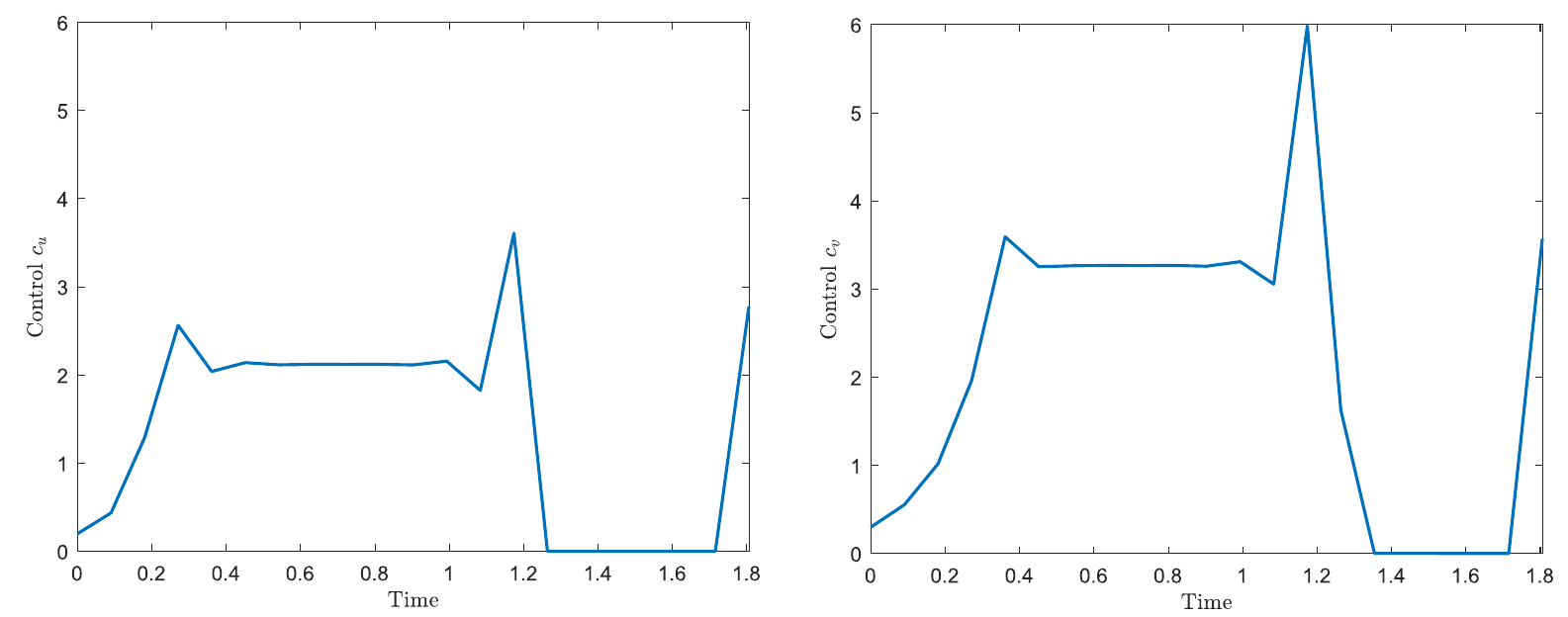}
   \captionsetup{font=small}
    \caption{Boundary controls $c_u$ and $c_v$.}
    \label{cucvh}
\end{figure}

\begin{figure}[h!]
        \centering
         \includegraphics[width=0.55\linewidth]{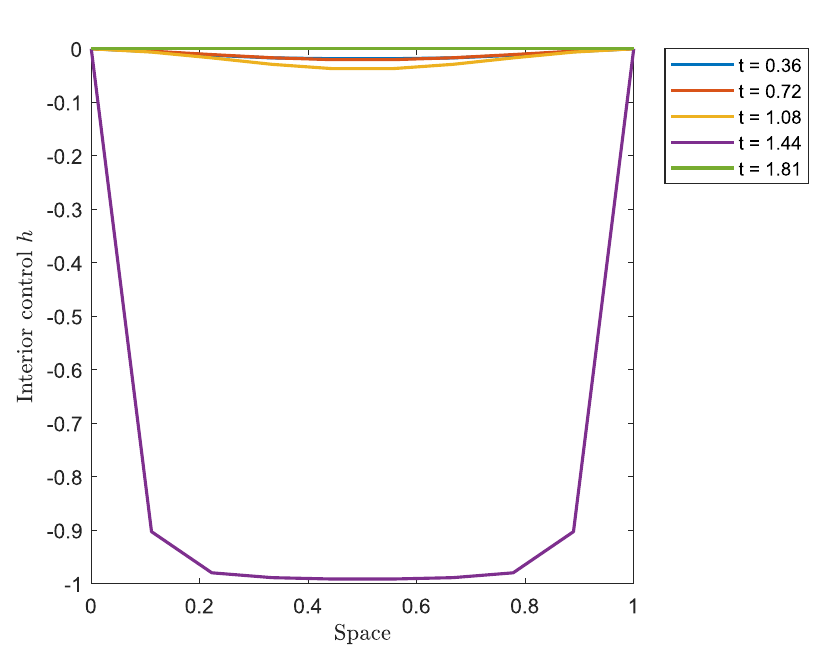}
   \captionsetup{font=small}
    \caption{Interior control $h$.}
    \label{h}
\end{figure}
In this case, we observe an interesting interplay between the boundary controls and the interior control in achieving the target in minimum time. Since the initial condition is above the desired target, \( h \) remains negative, reaching its lowest value at \( t = 1.44 \) before approaching zero as the trajectory nears the target. Meanwhile, \( c_u \) and \( c_v \) exhibit similar behavior on different scales, likely influenced by the initial conditions.

It is interesting to observe the initial behavior of the trajectories, which seems unrelated to the objective of approaching the target.
This reflects a dynamic rebalancing that the controls must achieve to drive the trajectory to a target that is positive in the interior of the domain and zero at the boundary.

The simulation of the same problem, but without considering the internal control $h$, results in a minimum time of $t=2.0083$. This shows that, in addition to allowing the target to be  reached in finite-time, the internal control accelerates the approach. The trajectories of $u$ and $v$, as well as the controls $c_u$ and $c_v$, can be seen in Figures \ref{uv} and \ref{cucv}, respectively.  

\begin{figure}[h!]
        \centering
         \includegraphics[width=1\linewidth]{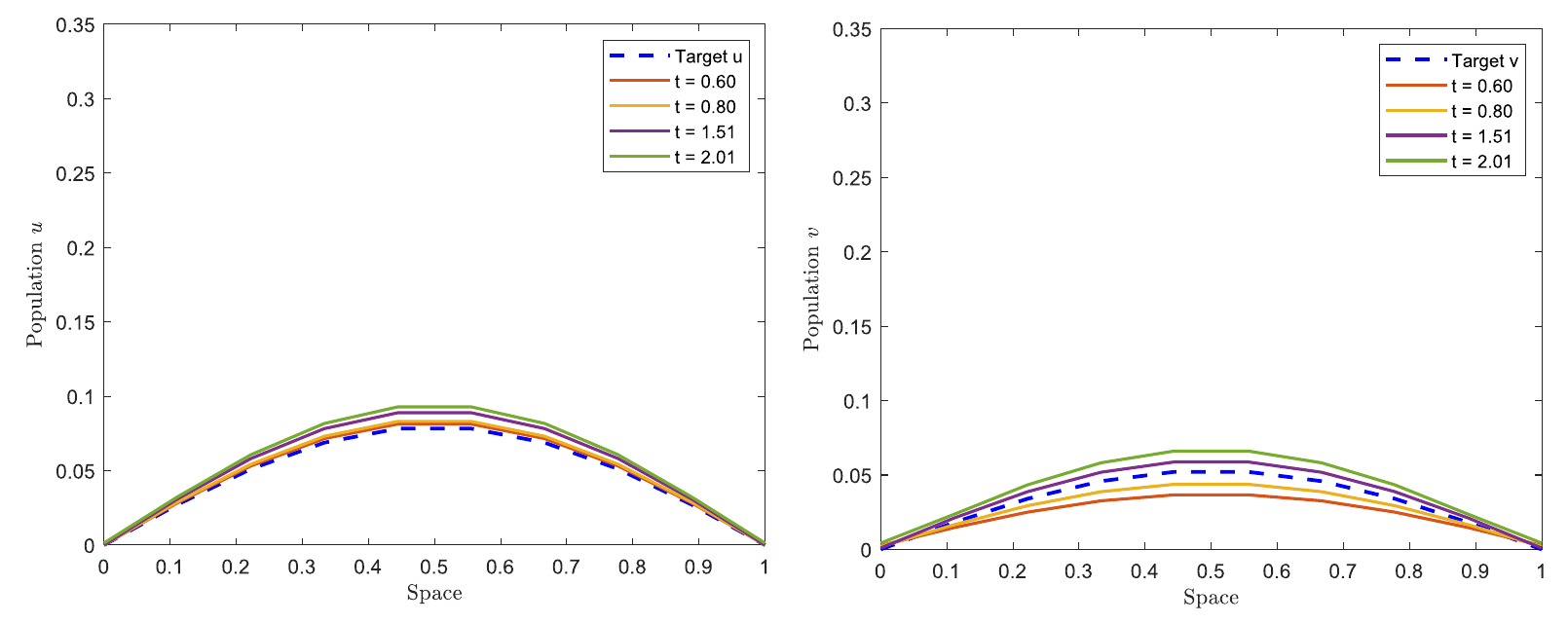}
   \captionsetup{font=small}
    \caption{Trajectories of $u$ and $v$ without interior control.}
    \label{uv}
\end{figure}

\begin{figure}[!h]
        \centering
         \includegraphics[width=1\linewidth]{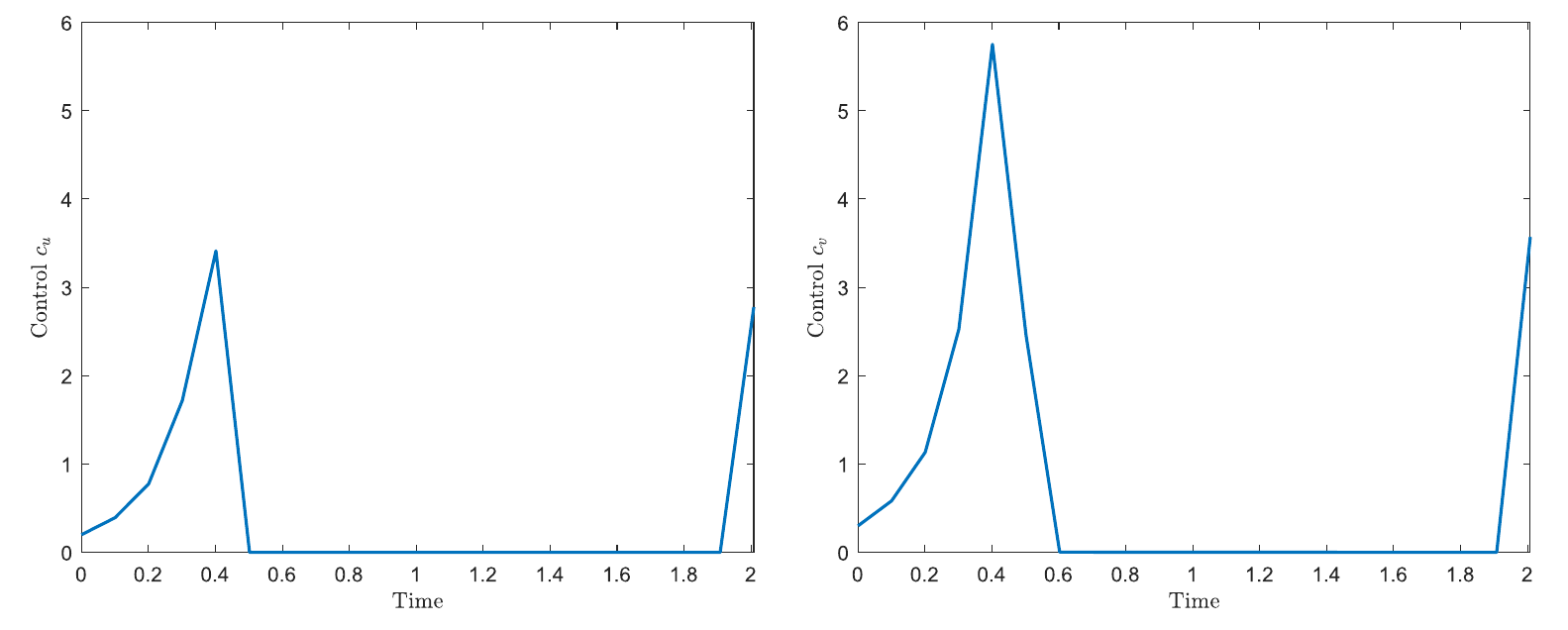}
   \captionsetup{font=small}
    \caption{Boundary controls $c_u$ and $c_v$ without interior control.}
    \label{cucv}
\end{figure}

In this case, we observe the trajectories and the controls on the boundary \( c_u \) and \( c_v \) in a more natural manner; however, as expected, with a minimum time greater than in the previous case.
\section{Concluding Remarks and Perspectives}

\subsection{Conclusions on our results} In this work, we analyze the global controllability of a general diffusive LVM describing the competition between two species in a smooth domain $\Omega \subset \mathbb{R}^N$ ($N=1,2,3$). The results, obtained by combining constrained boundary controls with a bounded multiplicative internal control, effectively support the biological interpretation of the model, confirming the relevance of the proposed control strategy.

Below, we provide a comprehensive analysis of the results for each considered target.
\begin{itemize}
\item {\it Survival of one of the species: $(a_1/b_1,0)$ and $(0,a_2/c_2)$}. This case is addressed in Theorem \ref{MT1}, and the asymptotic controllability is achieved through the combination of controls acting on the boundary and in the interior. More specifically, we provide boundary controls $c_u$ and $c_v$ satisfying \eqref{CUV}, along with a multiplicative interior control $hu$ such that for any initial condition $(u_0, v_0)$ ($0 \leq u_0 \leq a_1/b_1$ and $0 \leq v_0 \leq a_2/c_2$), the trajectory converges to the desired target as $t\to\infty$. This result holds independently of the problem parameters and the domain \( \Omega \). As mentioned earlier, achieving finite-time controllability for these targets is particularly challenging due to the considered constraints on controls and trajectories.
\item {\it  Heterogeneous coexistence : $(u^{**},v^{**})$}.
In this case, Theorem \ref{MT3} provides finite-time global  controllability; roughly speaking, the target is exactly reached at some time \( T \) from any initial condition satisfying the constraints. This case is particularly interesting since the target \( (u^{**}, v^{**}) \) is zero on the boundary, and thus finite-time controllability could require some oscillation of the trajectories, potentially leading to a violation of one of the constraints. However, the interior multiplicative control \( h1_{\omega}u \), combined with an appropriate control strategy, allows us to consider \( h \) sufficiently small so that comparison arguments can be used to ensure the assumed constraints. 

Without a doubt, Theorem \ref{MT3} is the main result of this work, as it provides a precise strategy to steer the trajectories toward a specific coexistence target. The inherent constraints of the problem make it more challenging to reach the target in finite-time, which was achieved only through the estimates obtained for the interior multiplicative control.

\item {\it Homogeneous coexistence and extinction: $(u^*,v^*)$ and $(0,0)$}. Although these cases are not the main focus of this work, a brief analysis is provided in Section \ref{HCE}. In particular, asymptotic controllability can always be achieved for the target \( (0,0) \) as long as multiplicative controls are used in both equations. Here, once again, finite-time controllability is more delicate due to the assumed constraints. Under certain conditions on the parameters, finite-time global  controllability for the target \( (u^*, v^*) \) can be achieved in a manner similar to what we did for the target \( (u^{**}, v^{**}) \). However, in this case, only boundary controls are sufficient (see \cite{SONEGOZUAZUA}).
\end{itemize}
\subsection{Other strategies for asymptotic control}\label{OCS}

It is possible to use other control strategies to asymptotically reach the targets $\left(a_1/b_1,0\right)$ and $\left(0,a_2/c_2\right)$. Below, we discuss boundary controls arising from the associated Neumann problem. 

Before, we state the following result, which can be found in \cite{Iida}.

\begin{theorem}\label{tpao}
    Let $(u,v)$ be the solution of 
    \begin{equation}\label{e1N}
\left\{\begin{array}{ll}
u_t=d_1\Delta u+u(a_1-b_1u-c_1v),& (x,t)\in\Omega\times\mathbb{R}^+\\
v_t=d_2\Delta v+v(a_2-b_2u-c_2v),& (x,t)\in\Omega\times\mathbb{R}^+\\
(u(x,0),v(x,0))=(u_0,v_0),& x\in\Omega\\
{\partial u}/{\partial\nu}=0,\ \ {\partial v}/{\partial\nu}=0& (x,t)\in\partial\Omega\times\mathbb{R}^+,
\end{array} \right.   
\end{equation}
where $\nu$ is the outward unit normal vector to $\partial\Omega$, $0\leq u_0\leq a_1/b_1$, $0\leq v_0\leq a_2/c_2$. Then $0\leq u(x,t)\leq a_1/b_1$, $0\leq v(x,t)\leq a_2/c_2$ for all $(x,t)\in\Omega\times\mathbb{R}^+$ and
\begin{enumerate}[(i)]
\item if ${a_1}/{a_2}<{b_1}/{b_2}$ and  ${a_1}/{a_2}<{c_1}/{c_2}$ then $\displaystyle\lim_{t\to\infty}(u(x,t),v(x,t))=(0,a_2/c_2)$ uniformly in $\Omega$;
\item if ${a_1}/{a_2}>{b_1}/{b_2}$ and  ${a_1}/{a_2}>{c_1}/{c_2}$ then $\displaystyle\lim_{t\to\infty}(u(x,t),v(x,t))=(a_1/b_1,0)$ uniformly in $\Omega$;
\item if ${c_1}/{c_2}<{a_1}/{a_2}<{b_1}/{b_2}$ then $\displaystyle\lim_{t\to\infty}(u(x,t),v(x,t))=(u^*,v^*)$ uniformly in $\Omega$.
\end{enumerate}
\end{theorem}

For example, if we assume an interior control $h\equiv \sigma$ in \eqref{e1} such that
\begin{equation}\label{sig}-a_1<\sigma<\min\{a_2b_1/b_2-a_1,a_2c_1/c_2-a_1\}\end{equation}
 we have $(i)$ satisfied with $a_1 + \sigma$ instead of $a_1$.
Given any initial condition $(u_0,v_0)$, we can consider the respective solution $(\widetilde{u},\widetilde{v})$ of \eqref{e1N} with $a_1 + \sigma$ instead of $a_1$. In this case, we assume the following constraints boundary controls
$$c_u=\widetilde{u}|_{\partial\Omega}, \ \ c_v=\widetilde{v}|_{\partial\Omega}.$$
 It follows from item $(i)$ of Theorem \ref{tpao} that the solution $(u,v)$ of \eqref{e1} converges to $(0,a_2/c_2)$ as $t\to\infty$, uniformly in $\Omega$. Note that, unlike Theorem \ref{MT1}, the boundary controls now depend on the variables \( x \) and \( t \).
  Obviously, the same can be done with respect to the target \((a_1/b_1,0)\), but now assuming a control in the second equation as in \eqref{e2s}.

  We observe that this control strategy is independent of the geometry of the domain and the diffusion capacity of the species. However, the interior control depends on all the other parameters that represent the interaction between species $u$ and $v$: $a_i$, $b_i$, $c_i$, for $i=1,2$ (see \eqref{sig}).
In this case, it was necessary to use boundary controls derived from the associated Neumann problem.

\subsection{Potential developments}
Firstly, with respect to Theorem \ref{MT1}, and consequently to the targets $(0, {a_2}{/c_2})$ and $({a_1}/{b_1}, 0)$, an issue that requires further analysis concerns what must be done in order to achieve the asymptotic controllability result when the control is activated only in a subset $\widetilde{\omega} \subset \Omega$. This is an interesting topic for further study, and it is likely that some assumptions on the coefficients will be necessary.

Now, regarding the homogeneous coexistence case, the strategy outlined in Subsection \ref{sub_u_v} can be applied to our system, provided that the parameters satisfy condition $(iii)$ of Theorem \ref{tpao}. Hence, in this case, the activation of an internal control is not required to steer the system toward the target $(u^*, v^*)$. On the other hand, if the conditions in $(iii)$ of Theorem \ref{tpao} are not fulfilled, it remains possible that an internal control could be employed to ensure asymptotic stabilization. However, this result has not been established in the present work and will certainly be the subject of future research.

As discussed in Subsection \ref{sub_0_0}, analyzing the extinction state $(0,0)$ under the influence of controls in one or both equations also remains an open topic for future studies. Although the methodology employed in this work is not directly applicable to the target \((0,0)\), due to the decoupling that occurs during the linearization process around this point, we believe that, although delicate, a future study worth considering would be to investigate the case of additive control, that is, the action of \( h1_{\omega} \) distributed in one of the equations. A relevant observation to highlight refers to the methodology employed by Coron, Guerrero, and Rosier \cite{CORONGUERREROROSIER}. In this study, the authors investigate the local null controllability of a system consisting of two parabolic equations. The system in question includes a forced control term in one of the equations and a cubic coupling term in the other. The system is described as follows:
\begin{equation}\label{system_of_Coron}
\left\{\begin{array}{ll}{u}_t - \Delta {y} = g(u,v) +  h1_{\omega},& (x,t)\in Q\\
{v}_t - \Delta {v} = u^{3} + R v,& (x,t)\in Q\\
(u(x,0),v(x,0))=(u_0,v_0),& x\in\Omega\\
{u}(x,t)= 0,\ \ {v}(x,t)= 0,& (x,t)\in \Sigma,
\end{array} \right.   
\end{equation}
where \( g : \mathbb{R} \times \mathbb{R} \to \mathbb{R} \) is a given function of class \( C^\infty \) vanishing at \( (0, 0) \in \mathbb{R} \times \mathbb{R} \), and \( R \) is a given real number. The linearization of this system around zero also decoupled the system, just like in our target $(0,0)$, causing the control \( h \) to have no influence on \( v \), and if \( v(., 0) \neq 0 \), then \( v(.,T) \neq 0 \). Thus, the authors introduced the return method (see \cite{CORON1992} and \cite{CORON1996} for further details), a constructive approach that is highly sensitive and challenging to implement. This method involves selecting a trajectory \( ((\bar{u}, \bar{v}), \bar{h}) \) for the previous control system that satisfies the following conditions:
\begin{itemize}
   \item[(i)] it goes from \( (0, 0) \) to \( (0, 0) \), i.e., \( u(0, \cdot) = v(0, \cdot) = u(T, \cdot) = v(T, \cdot) = 0 \);
   \item[(ii)] the linearized control system around that trajectory is null controllable.
\end{itemize}
Using this trajectory and a suitable fixed point argument, the authors established that if there exists \( \delta > 0 \) such that for \( (u_0, v_0) \in L^{\infty}(\Omega) \times L^{\infty}(\Omega) \) with \( \| u_0 \|_{L^{\infty}(\Omega)} + \| v_0 \|_{L^{\infty}(\Omega)} < \delta \), then there exists a null control \( h \in L^{\infty}(\omega \times (0,T)) \) for the system \eqref{system_of_Coron}. However, Coron, Guerrero, and Rosier concluded that, due to the strong maximum principle, it is not possible to obtain a similar result if \( u^3 \) is replaced by \( u^2 \), unless functions with complex values are considered.

That being said, one of our future research interests would be to investigate whether it is possible to implement the ideas from \cite{CORONGUERREROROSIER} for our diffusive LVM in order to obtain new controllability results, including a positive outcome for the target \( (0, 0) \).

Considering the results presented,  new directions for research on system \eqref{e1} and/or \eqref{e1_2} naturally arise, particularly with modifications that could alter its dynamics. One such potential modification involves adapting the system to represent a mutualistic scenario. For instance, if in \eqref{e1}, we define \( a_1, a_2 > 0 \), \( b_1 > 0 \), \( b_2 < 0 \), \( c_1 < 0 \), and \( c_2 > 0 \), then both populations benefit mutually, meaning that \( u \) and \( v \) support each other's growth, resulting in a cooperative dynamic. The growth of one population enhances the growth of the other. 

Another widely discussed possibility is allowing all system parameters to depend on the spatial and/or temporal variables. Establishing global controllability results for this extended model is a relevant objective, as it would generalize our framework \cite{A1,ACD,HE,NI}.

\appendix
\renewcommand\theequation{\Alph{section}.\arabic{equation}}
\makeatletter
\@addtoreset{equation}{section}
\makeatother

\section{Proof of Proposition \ref{cnsl}}\label{Appendix}

\begin{proof}
Following the arguments of \cite{Clark}, for each \( n \geq 1 \), define the functions
\begin{equation*}
    \begin{array}{ccc}
&\beta_n := {\beta(T - t)} / {[(T - t) + {1}/{n}]},&  \gamma_n := {\gamma (T - t)} / {[(T-t) + {1}/{n}]},\\ 
&\rho_n := e^{s\beta_n},&  \rho_{0,n} := \rho_n \gamma_{n}^{-3/2}
 \end{array}
\end{equation*}
and 
\begin{equation*}
    \begin{array}{ll}
\bar{\rho}_{n} = e^{s\beta}\gamma^{-7/2}.\chi_{n},\quad \text{where}\quad \chi_{n}=\left\{\begin{array}{l}
 1 \,\,\, \text{in}\,\,\, \omega_0\\
 n\,\,\, \text{in}\,\,\,\Omega\setminus\overline{\omega}_0
\end{array}\right.
     \end{array}
\end{equation*}
and the functional $J_n:L^{2}(Q)\times L^{2}(Q) \times L^{2}(\omega_0\times (0,T))\rightarrow\mathbb{R}$, of the form
\begin{equation*}
    \begin{array}{l}
         J_{n}(y,z,\widetilde{h}):=\dfrac{1}{2}\displaystyle\iint_{Q}\left[\left(\rho^{2}_{0,n}|y|^{2} + \rho^{2}_{n}|z|^{2} \right) + \bar{\rho}^{2}_{n}|\widetilde{h}|^{2}\right]dxdt.
    \end{array}
\end{equation*}
The central idea is to address the following extremal problem:
\begin{equation*}
    \left\{\begin{array}{l}
        \text{Minimize}\quad J_{n}(y,z,\widetilde{h})\\
        \text{subject to}\quad \widetilde{h}\in L^{2}(\omega_0\times (0,T),\quad (y,z,\widetilde{h}) \quad \text{satisfies}\quad \eqref{linear_1}. 
    \end{array}\right.
\end{equation*}
It is proven that \( J_n \) is a lower semicontinuous, convex, and coercive functional. Therefore, it has a minimum $(y_{n},z_n,\widetilde{h}_{n})$. Moreover, based on Lagrange's Principle, there exist \((p_n, q_n)\) such that the pairs \((y_n, z_n)\), \((p_n, q_n)\), and \(\widetilde{h}_n\) satisfy:
\begin{equation}\label{linear_de_yn}
\left\{\begin{array}{ll}{y}_{n,t}=d\Delta {y}_n+{y}_n(a-2b_1u^{**}- c_1 v^{**}) - c_1u^{**} z_n  + {\widetilde{h}_n}1_{\omega_0} + F_0,& (x,t)\in Q\\
{z}_{n,t}=d\Delta {z}_n + {z}_n(a - b_2 u^{**} - 2c_2v^{**} ) - b_2v^{**} y_n + F_1,& (x,t)\in Q\\
(y_n(x,0),z_n(x,0))=(y_0,z_0),& x\in\Omega\\
{y}_n(x,t)= 0,\ \ {z}_n(x,t)= 0,& (x,t)\in\Sigma,
\end{array} \right.   
\end{equation}

\begin{equation}\label{adjoint_de_yn}
    \left\{\begin{array}{ll}
         -p_{n,t} = d\Delta p_n + p_n\left(a - 2b_1 u^{**} - c_1v^{**}\right)- b_{2}v^{**} q_n - \rho^{2}_{0,n}y_n, & (x,t)\in Q\\
         -q_{n,t} = d\Delta q_n + q_n\left( a - b_2 u^{**} - 2c_{2}v^{**} \right) - c_1 u^{**}p_n  - \rho^{2}_{n}z_n,  & (x,t)\in Q\\
         (p_n(x,T),q_n(x,T) )= (0,0),   & x\in \Omega\\
         p_n(x,t) = 0,\,\,\,  q_n(x,t) = 0, & (x,t)\in\Sigma,
    \end{array}\right.
\end{equation}
\begin{equation}\label{hn}
    p_{n} = \bar{\rho}^{2}_n\widetilde{h}_{n},\quad \text{in}\quad \omega_0\times (0,T). 
\end{equation}

Multiplying the first equation of \eqref{adjoint_de_yn} by \( y_n \) and the second by \( z_n \), integrating over \( Q \), and performing integration by parts, we obtain
\begin{equation}\label{igual_0}
    \begin{array}{l}
    \displaystyle\iint_{Q}(p_n,q_n)\cdot (\widetilde{h}_n1_{\omega_0} + F_0, F_1)dxdt + \displaystyle\iint_{Q}(\rho^2_{0,n}|y_n|^{2} + \rho^2_{n}|z_n|^{2} ) dxdt\\
         +\displaystyle\int_{\Omega}(p_{n}(x,0)y_0 + q_n(x,0)z_0)dx = 0.
    \end{array}
\end{equation}
From $J_n$ and \eqref{igual_0}, and considering \eqref{hn}, we find that
\begin{equation*}
    \begin{array}{l} J_n(y_n,z_n,\widetilde{h}_n) = -\dfrac{1}{2}\displaystyle\iint_{Q}(p_n F_0 + q_n F_1)dxdt -\dfrac{1}{2}\displaystyle\int_{\Omega}(p_n(x,0)y_0 + q_n(x,0)z_0)dx\\

    \leq  C\left[\Vert p_n(.,0)\Vert_{L^{2}(\Omega)} ^{2} + \Vert q_n(.,0)\Vert_{L^{2}(\Omega)}^{2} + \displaystyle\int_{Q} e^{-2s\beta}\gamma^{3}(|p_n|^{2} + |q_n|^{2})dxdt\right]^{1/2}\\
    \cdot\left[ \Vert y_0\Vert_{L^{2}(\Omega)}^{2} + \Vert z_0\Vert_{L^{2}(\Omega)}^{2} + \displaystyle\iint_{Q}e^{2s\beta}\gamma^{-3}(|F_0|^{2} + |F_1|^{2})dxdt \right]^{1/2}.
    \end{array}
\end{equation*}

Applying Proposition \ref{Carleman_inicial} to \( (p_n, q_n) \), the solution of \eqref{adjoint_de_yn}, and using the fact that \( \rho_n \leq e^{s\beta} \), \( \rho_{0,n} \leq e^{s\beta}\gamma^{-3/2} \), and the definition of \( \bar{\rho_n} \), we get
\begin{equation*}
    \begin{array}{ll}
\displaystyle\iint_{Q}e^{-2s\beta}\gamma^{3}(|p_n|^{2} + |q_n|^{2})dxdt \leq  \displaystyle\iint_{Q}(e^{-2s\beta}\gamma^{3}\rho^{4}_{0,n}|y_n|^{2} + e^{-2s\beta}\rho_{n}^{4}|z_n|^{2})dxdt \\
\quad  + \displaystyle\iint_{\omega_0\times(0,T)}e^{-2s\beta}\gamma^{7}|p_n|
^{2}dxdt    \\
\leq  C\left[\displaystyle\iint_{Q}(\rho^{2}_{0,n}|y_n|^{2} + \rho^{2}_{n}|z_n|^{2})dxdt + \displaystyle\iint_{\omega_{0}\times(0,T)}e^{2s\beta}\gamma^{-7}|\widetilde{h}_n|^{2}dxdt\right]\\
\leq 
 C J_{n}(y_n,z_n,\widetilde{h}_n).
\end{array}
\end{equation*}
Additionally, by applying the standard energy estimate in \eqref{adjoint_de_yn}, we obtain
\begin{equation*}
    \Vert p_n(.,0)\Vert_{L^{2}(\Omega)}^{2} + \Vert q_n(.,0)\Vert_{L^{2}(\Omega)}^{2}\leq C J_{n}(y_n,z_n,\widetilde{h}_{n}).
\end{equation*}
Therefore, from the two previous inequalities, we can conclude 
\begin{equation*}
    J_{n}(y_n,z_n,\widetilde{h}_{n})\leq C\left[ \Vert y_0\Vert_{L^{2}(\Omega)}^{2} + \Vert z_0\Vert_{L^{2}(\Omega)}^{2} + \displaystyle\iint_{Q}e^{2s\beta}\gamma^{-3}(|F_0|^{2} + |F_1|^{2})dxdt\right]
\end{equation*}
and consequently, one obtains that
\begin{equation}\label{estimativa_para_y_z_e_h}
    \begin{array}{l}
\displaystyle\iint_{Q}\left(\rho^{2}_{0,n}|y_n|^{2} + \rho^{2}_{n}|z_n|^{2}\right)dxdt + \displaystyle\iint_{\omega_0\times (0,T)}\bar{\rho}^{2}_{n}|\widetilde{h}_{n}|^{2}dxdt   \\
\leq C\left[ \Vert y_0\Vert^{2}_{L^{2}(\Omega)} + \Vert z_0\Vert^{2}_{L^{2}(\Omega)} + \displaystyle\iint_{Q}e^{2s\beta}\gamma^{-3}\left(|F_0|^{2} + |F_1|^{2}\right)dxdt\right]\\ 
       = C\kappa(y_0,z_0,F_0,F_1).  
    \end{array}
\end{equation}

Now, we will show that \( \widetilde{h}_n \in L^{\infty}(\omega_0\times (0,T)) \). Let \( \xi > 0 \) be sufficiently small, and let \( \left\{ \xi_j \right\}_{j=0}^{M} \) be a finite increasing sequence such that $0<\xi_{j}<\xi$, $j=0,1,\ldots, M-1, \xi_{M}=\xi$. Since $N\leq 3$, let also $\left\{ r_j \right\}_{j=0}^{M}$ another finite increasing sequence such that $r_0 = 2, r_{M}=\infty$ and, for $j=0,1,\ldots, M-1$,
\begin{equation*}
    -(N/2 + 1)(1/r_{j}-1/r_{j+1}) + 1 > 1/2.
\end{equation*}
Set 
\begin{equation*}
    \begin{array}{l}
         \alpha_{n}:= - \beta_{n} =  {(e^{\lambda\eta(x)}-e^{R\lambda})(T-t)} / \lbrace {m(t)[(T - t) + {1}/{n}]}\rbrace 
            \end{array}
\end{equation*}
and
\begin{equation*}
    \begin{array}{l}
     \alpha_{0n}:= {(1-e^{R\lambda})(T-t)} / \lbrace{m(t)[(T-t)+1/n]}\rbrace. 
    \end{array}
\end{equation*}
Then, 
\begin{equation}\label{desigualdade_de_alpha_n}
    \alpha_{0n}\leq \alpha_{n}\leq {\alpha_{0n}} / {(1+e^{-R\lambda})}<0.
\end{equation}
For each $j$, define 
\begin{equation*}
   \left\{ \begin{array}{ll}
         \phi_{j}(x,t) = e^{(s+\xi_{j})\alpha_{0n}}\,p_n(x,T-t); & \vartheta_{j}(x,t) = e^{(s+\xi_{j})\alpha_{0n}}\,q_n(x,T-t);\\
         g_{j}(x,t) = (e^{(s+\xi_{j})\alpha_{0n}})_{t}\,p_n(x,T-t);& k_{j}(x,t) = (e^{(s+\xi_{j})\alpha_{0n}})_{t}\,q_n(x,T-t);\\
         y_{j}(x,t) = e^{(s+\xi_{j})\alpha_{0n}}\rho_{0,n}^{2}\,y_{n}(x,T-t);& z_{j}(x,t) = e^{(s+\xi_{j})\alpha_{0n}}\rho_{n}^{2}\,z_n(x,T-t).
    \end{array}\right.
\end{equation*}
Thus, from equation \eqref{adjoint_de_yn}, we can conclude that, for each \( j \), the pair \( (\phi_j, \vartheta_j) \) is the solution of the following system
\begin{equation}\label{eq_de_phi_e_theta}
    \left\{\begin{array}{ll}
         \phi_{j,t} = d\Delta \phi_{j} + \phi_j\left(a - 2b_1 u^{**} - c_1v^{**}\right)- b_{2}v^{**} \vartheta_j - y_j + g_{j}, & (x,t)\in Q\\
         \vartheta_{j,t} = d\Delta \vartheta_{j} + \vartheta_{j}\left( a - b_2 u^{**} - 2c_{2}v^{**} \right) - c_1 u^{**}\phi_{j}  - z_j + k_{j},  & (x,t)\in Q\\
         (\phi_{j}(x,0),\vartheta_{j}(x,0) )= (0,0),   & x\in \Omega\\
         \phi_j(x,t) = 0,\,\,\,  \vartheta_j(x,t) = 0, & (x,t)\in\Sigma.
    \end{array}\right.
\end{equation}
Consider the semigroup $\{S(t); t \geq 0\}$ generated by the heat equation with Dirichlet boundary conditions. For $\psi \in L^p(\Omega)$, it holds that
$$
\Vert S(t)\psi\Vert_{L^{q}} \leq C \max \left\{ t^{-\frac{N}{2} \left( \frac{1}{p} - \frac{1}{q} \right) }, 1 \right\} \Vert\psi\Vert_{L^{p}},
$$
for all $0 < t < \infty$, with $1 \leq p < q \leq \infty$ (see, \cite{ARENDT,DAVIES}). Then,  invoking this inequality in the context of the solution to the system \eqref{eq_de_phi_e_theta}, and after performing the necessary calculations, we derive that 
\begin{equation}
    \Vert e^{(s+\xi)\alpha_{0n}}p_n\Vert^{2}_{L^{\infty}(Q)}\leq C \kappa(y_0,z_0,F_{0},F_1).
\end{equation}
Therefore, by \eqref{hn} and using the fact that $e^{s\beta}\geq \rho_{n}$ and $\alpha_n = -\beta_{n}$ we have that
\begin{equation*}
    \Vert e^{[-s(1-e^{-\lambda R})+\xi(1+e^{-\lambda R})]\alpha_{n}}\widetilde{h}_{n} \Vert^{2}_{L^{\infty}(\omega_0\times(0,T))}\leq C \kappa(y_0,z_0,F_{0},F_1). 
\end{equation*}
Consequently, by selecting $\xi$ to be sufficiently small, it follows that $$-s(1-e^{-\lambda R})+\xi(1+e^{-\lambda R})<0$$ and so
\begin{equation}\label{estimativa_para_hn}
    \Vert \widetilde{h}_{n}\Vert^{2}_{L^{\infty}(\omega_0\times (0,T))} \leq C \kappa(y_0,z_0,F_{0},F_1).
\end{equation}

By \eqref{estimativa_para_y_z_e_h} and \eqref{estimativa_para_hn} we can extract suitable subsequences (again indexed by $n$) such that, by the definitions of $\rho_n, \rho_{0,n}$ and $\bar{\rho}_n$ satisfy 
\begin{equation*}
   \left\{ \begin{array}{lll}
    \rho_{0,n}y_n\rightarrow e^{s\beta}\gamma^{-3/2} y & \text{weak in} & L^{2}(Q), \\
\rho_{n}z_n\rightarrow\ e^{s\beta} z & \text{weak in} & L^{2}(Q),\\
\bar{\rho}_{n}\widetilde{h}_n\rightarrow e^{s\beta}\gamma^{-7/2} \widetilde{h}& \text{weak in} & L^{2}(\omega_0\times(0,T)),\\
\widetilde{h}_{n}\rightarrow\widetilde{h} & \text{weak* in} & L^{\infty}(\omega_0\times(0,T)).
    \end{array}\right.
\end{equation*}
Hence, by taking the limit of the linear system \eqref{linear_de_yn}, we can deduce that $(y, z)$ represents the state of the system \eqref{linear_1} associated with $\widetilde{h}$, such that \eqref{regularidade_inicial_pesos_que_dependem_de_x_e_t} holds.

To get the equation \eqref{regularidade_a_mais_com_pesos_que_dependem_de_x_e_t}, one must multiply the first equation of \eqref{linear_1} by the factor \( e^{2s\beta} \gamma^{-5}y \), and the second equation by the factor \( e^{2s\beta} \gamma^{-5}z \), then proceed with the integration over \( \Omega \). So, after performing some calculations and analyzing the weights, the expression is obtained
\begin{equation*}
    \begin{array}{l}
         \dfrac{d}{dt}\displaystyle\int_{\Omega}e^{2s\beta}\gamma^{-5}  (|y|^{2} + |z|^{2})dx + \displaystyle\int_{\Omega}e^{2s\beta}\gamma^{-5} (d|\nabla y|^{2} + d|\nabla z|^{2})dx  \\
         \leq C\left(\displaystyle\int_{\Omega} e^{2s\beta}(\gamma^{-3}|y|^{2} + |z|^{2})dx + \displaystyle\int_{\omega_0}e^{2s\beta}\gamma^{-7}|\widetilde{h}|^{2}dx\right.\\
         \left.+ \displaystyle\int_{\Omega} e^{2s\beta}\gamma^{-3}(|F_0|^{2} + |F_1|^{2})dx\right).
    \end{array}
\end{equation*}
Thus, by integrating over time, the estimate \eqref{regularidade_a_mais_com_pesos_que_dependem_de_x_e_t} is achieved.
\end{proof}
\setcounter{equation}{0}
\section{Study of map $\mathcal{A}$}\label{Appendix_B}
This appendix is devoted to verifying that the mapping $\mathcal{A}:\mathcal{E}\rightarrow\mathcal{Z}$, see  beginning of Section \ref{Sec4}, defined by
\begin{equation*}
\begin{array}{c}
    \mathcal{A}(y,z,\widetilde{h}) 
    :=
    \Big( {y}_t - d\Delta {y} - {y}(a-b_1(2{u}^{**} + y) - c_1z - c_1 {v}^{**}) + c_1{u}^{**} z\\
    - \widetilde{h}1_{\omega_0}(\frac{y}{u^{**}}+1),
         {z}_t -d\Delta {z} - {z}(a-c_2(2{v}^{**} + z) - b_2{y}- b_2 {u}^{**}) + b_2{v}^{**} y,\\
         y(.,0),\, z(.,0) \Big),
    \end{array}
\end{equation*}
satisfies the assumptions of Theorem~\ref{Liusternik}. To this end, we first observe, by virtue of Proposition \ref{control_of_SL}, that
\begin{equation}\label{resultado_da_proposi_o_10}
    \begin{array}{l}
\displaystyle\iint_{Q}e^{2s{\hat{\beta}}}\left(\hat{{\gamma}}^{-3}\vert y\vert^{2} + \vert z \vert^{2}\right)dxdt +  \displaystyle\iint_{\omega_0\times (0,T)}e^{2s{\hat{\beta}}}{\hat{\gamma}}^{-7}\vert \widetilde{h}\vert^{2}dxdt\\
+ \,\displaystyle\sup_{[0,T]}\displaystyle\int_{\Omega}e^{2s{\hat{\beta}}}{\hat{\gamma}}^{-5} ( \vert  y\vert^{2} +  \vert  z\vert^{2}) dx \leq C\, \Vert (y,z,\widetilde{h})\Vert^{2}_{\mathcal{E}}.
    \end{array}
\end{equation}

Let us break down $\mathcal{A}$ as follows:
\begin{equation*}
\mathcal{A}(y,z,h):= \left( \mathcal{A}_1(y,z,h), \mathcal{A}_2(y,z,h), \mathcal{A}_3(y,z,h), \mathcal{A}_4(y,z,h)\right),
\end{equation*}
where
\begin{equation}\label{Mapas_Ai}
    \left\{\begin{array}{l}
         \mathcal{A}_1(y,z,\widetilde{h}) := {y}_t - d\Delta {y} - {y}(a-b_1(2{u}^{**} + y) - c_1z - c_1 {v}^{**}) + c_1{u}^{**} z \\
         - \, \widetilde{h}1_{\omega_0}(\frac{y}{u^{**}}+1);  \\
         \mathcal{A}_2(y,z,\widetilde{h}) :=  {z}_t -d\Delta {z} - {z}(a-c_2(2{v}^{**} + z) - b_2{y}- b_2 {u}^{**}) + b_2{v}^{**} y;  \\
         \mathcal{A}_3(y,z,\widetilde{h}) := y(.,0);  \\
         \mathcal{A}_4(y,z,\widetilde{h}) := z(.,0),
    \end{array}\right.
\end{equation}
for all $(y,z,h)\in \mathcal{E}$. Note that,
\begin{equation*}
    \begin{array}{l}
        \Vert\mathcal{A}_{1}(y,z,\widetilde{h})\Vert_{\mathcal{U}}^{2}\\
        =\displaystyle\iint_{Q}e^{2s\beta^{*}}(\gamma^{*})^{-3}\Big[\vert{y}_t - d\Delta {y} - {y}(a-b_1(2{u}^{**} + y) - c_1z - c_1 {v}^{**}) \\
        + c_1{u}^{**} z
        - \widetilde{h}1_{\omega_0}(\frac{y}{u^{**}}+1)\vert^{2}\Big]dxdt\\
         \leq \displaystyle\iint_{Q}e^{2s\beta^{*}}(\gamma^{*})^{-3}\vert {y}_t - d\Delta {y} - {y}(a-2b_1u^{**}- c_1 v^{**}) + c_1u^{**} z  - {\widetilde{h}}1_{\omega_0}\vert^{2}dxdt \\
         
         +  \displaystyle\iint_{Q}e^{2s\beta^{*}}(\gamma^{*})^{-3} c_1^{2}\vert y\vert^{2}\vert z\vert^{2}dxdt
         
      +   \displaystyle\iint_{Q}e^{2s\beta^{*}}(\gamma^{*})^{-3} b_1^{2}\vert y\vert^{4}dxdt\\
        
           +   \displaystyle\iint_{\omega_0\times (0,T)}e^{2s\beta^{*}}(\gamma^{*})^{-3}\dfrac{1}{|u^{**}|^{2}}|\widetilde{h}y|^{2}dxdt\\
        
        \leq  C\left( \displaystyle\iint_{Q}e^{2s\beta^{*}}(\gamma^{*})^{-3}\vert F_0\vert^{2}dxdt\right.   \\
        + \displaystyle\iint_{Q}e^{2s(\beta^{*}-2\hat{\beta})}(\gamma^{*})^{-3}\hat{\gamma}^{8} e^{2s\hat{\beta}}\hat{\gamma}^{-3}\vert y\vert^{2} e^{2s\hat{\beta}}\hat{\gamma}^{-5}\vert z\vert^{2}dxdt\\
      
    + \displaystyle\iint_{Q}e^{2s(\beta^{*}-2\hat{\beta})}(\gamma^{*})^{-3}\hat{\gamma}^{8} e^{2s\hat{\beta}}\hat{\gamma}^{-5}e^{2s\hat{\beta}}\hat{\gamma}^{-3} \vert y\vert^{4}dxdt \\

        \left. + \displaystyle\iint_{\omega_0\times (0,T)}e^{2s(\beta^{*}-2\hat{\beta})}(\gamma^{*})^{-3}\hat{\gamma}^{12} e^{2s\hat{\beta}}\hat{\gamma}^{-5}\vert y\vert^{2} e^{2s\hat{\beta}}\hat{\gamma}^{-7}\vert \widetilde{h}\vert^{2}dxdt \right).
\end{array}
\end{equation*}
Then,  by estimate \eqref{resultado_da_proposi_o_10} and the condition \eqref{desigualdade_dos_beta} we have 
\begin{equation}
 \begin{array}{l}\label{A1_bem_definido}
        \Vert\mathcal{A}_{1}(y,z,\widetilde{h})\Vert_{\mathcal{U}}^{2} \leq  C\left( \displaystyle\iint_{Q}e^{2s\beta^{*}}(\gamma^{*})^{-3}\vert F_0\vert^{2}dxdt\right.\\
        \left.+ \displaystyle\sup_{[0,T]}\displaystyle\int_{\Omega}e^{2s\hat{\beta}}\hat{\gamma}^{-5}(\vert y\vert^{2} + \vert z\vert^{2})dx\displaystyle\iint_{Q} e^{2s\hat{\beta}}\hat{\gamma}^{-3}\vert y\vert^{2} dxdt \right.\\

       \left.+\, \displaystyle\sup_{[0,T]}\displaystyle\int_{\Omega}e^{2s\hat{\beta}}\hat{\gamma}^{-5}\vert y\vert^{2}dx\displaystyle\iint_{\omega_0\times (0,T)} e^{2s\hat{\beta}}\hat{\gamma}^{-7}|\widetilde{h}|^{2}dxdt   \right)\\
        
        \leq C ( 1 + \Vert (y,z,\widetilde{h})\Vert_{\mathcal{E}}^{2} ) \Vert (y,z,\widetilde{h})\Vert_{\mathcal{E}}^{2},     
    \end{array}
\end{equation}
for any $(y,z,\widetilde{h})\in\mathcal{E}$. Consequently, \( \mathcal{A}_1 \) is well-defined. The analysis for \( \mathcal{A}_2 \) is done in a similar manner and for \( \mathcal{A}_3 \) and \( \mathcal{A}_4 \) it follows trivially. The continuity of \( \mathcal{A} \) follows in a similar manner.

Now, let us fix \((y,z,\widetilde{h}) \in \mathcal{E}\) and take \((y',z',\widetilde{h}') \in \mathcal{E}\), along with \(\sigma > 0\). According to the decomposition presented in equation \eqref{Mapas_Ai}, we define the linear mapping \(\mathcal{DA}: \mathcal{E} \to \mathcal{Z}\), with \(\mathcal{DA}(y,z,\widetilde{h}) = \mathcal{DA} = (\mathcal{DA}_1, \mathcal{DA}_2, \mathcal{DA}_3, \mathcal{DA}_4)\) where
\begin{equation}\label{DAi}
 \left\{ \begin{array}{lll}
 \mathcal{DA}_{1}(y',z',\widetilde{h}') := y'_{t}-d\Delta y' - y'[a - 2b_1(u^{**} + y) - c_1(v^{**} + z) ]  \\
 +\, c_1 z'(u^{**} + y) - \widetilde{h}1_{\omega_{0}}\frac{ y'}{u^{**}} - \widetilde{h}'1_{\omega_{0}}(\frac{y}{u^{**}} + 1);\\
 
  \mathcal{DA}_{2}(y',z',\widetilde{h}'):= z'_{t} - d\Delta z' - z'[a - b_2(u^{**} + y) - 2c_2(v^{**} + z)]  \\
  +\, b_2y'(v^{**} + z);\\
  
  \mathcal{DA}_{3}(y',z',\widetilde{h}'):= y'(.,0);\\

  \mathcal{DA}_{4}(y',z',\widetilde{h}'):=z'(.,0).
    \end{array}\right.
\end{equation}
From the definitions of the spaces \(\mathcal{E}\) and \(\mathcal{Z}\), along with equation \eqref{DAi}, it is evident that \(\mathcal{DA} \in \mathcal{L}(\mathcal{E}, \mathcal{Z})\). Moreover, for each \(j = \{1, 2, 3, 4\}\), using arguments analogous to those applied in the well-definition of \(\mathcal{A}_1\), we obtain the following convergence:

\begin{equation}\label{Converg_de_Ai}
\begin{array}{l}
    {1} / {\sigma} \left[\mathcal{A}_{j}\left((y,z,\widetilde{h}) + \sigma(y',z',\widetilde{h}')\right) - \mathcal{A}_{j}(y,z,\widetilde{h}) \right]\rightarrow \mathcal{DA}_{j}(y',z',\widetilde{h}')  \\
    \, \text{in their respective spaces} \, \text{as} \, \sigma \to 0,\,\, \text{for each}\,\, j=\lbrace 1, 2, 3, 4\rbrace.
\end{array}
\end{equation}
This leads us to conclude that \( \mathcal{A} \) is Gâteaux-differentiable at any \( (y, z, \widetilde{h}) \in \mathcal{E} \), with G-derivative \( \mathcal{A}'(y, z, \widetilde{h}) = \mathcal{DA}(y, z, \widetilde{h}) \). Additionally, it is shown through standard arguments that \( \mathcal{A}' \) is continuous, and consequently, \( \mathcal{A} \) is not only Gâteaux-differentiable but also Fréchet-differentiable and \( C^1 \).

In fact, take $(y,z,\widetilde{h})\in\mathcal{E}$ and let $((y_n,z_n,\widetilde{h}_n))_{n\in\mathbb{N}}$ be a sequence such that $\Vert(y_n,z_n,\widetilde{h}_n) -(y,z,\widetilde{h})\Vert_{\mathcal{E}}\rightarrow 0$. Let us prove that 
\begin{equation}\label{convergencia_da_derivada}
    \Vert (\mathcal{A}'(y_n,z_n,\widetilde{h}_n) - \mathcal{A}'(y,z,\widetilde{h}))(y', z', \widetilde{h}') \Vert_{\mathcal{Z}} \rightarrow 0.
\end{equation}
In order to simplify the notation, we will consider
\begin{equation*}
\mathcal{D}_{j,n}:=\mathcal{A}'_{j}(y_n,z_n,\widetilde{h}_n)-\mathcal{A}'_{j}(y,z,\widetilde{h}).
\end{equation*}
Thus, using the same arguments as \eqref{A1_bem_definido}, we conclude that
\begin{equation*}
\begin{array}{c}
  \Vert\mathcal{D}_{1,n}(y', z', \widetilde{h}')\Vert^{2}_{\mathcal{U}}
 \leq  C\Big(\Vert y'(y_n-y)\Vert^{2}_{\mathcal{U}} + \Vert y'(z_n - z)\Vert^{2}_{\mathcal{U}}\\
 + \Vert z'(y_n - y)\Vert^{2}_{\mathcal{U}} + \Vert y'(\widetilde{h}_{n}1_{\omega_{0}} - \widetilde{h}1_{\omega_0})\Vert^{2}_{\mathcal{U}} + \Vert \widetilde{h}1_{\omega_0}(y_n - y)\Vert^{2}_{\mathcal{U}}\Big)\\

\leq  C \Vert (y_n,z_n,\widetilde{h}_n) - (y,z,\widetilde{h})\Vert^{2}_{\mathcal{E}} \Vert (y,z,\widetilde{h})\Vert^{2}_{\mathcal{E}}  \Vert (y',z',\widetilde{h}')\Vert^{2}_{\mathcal{E}}\rightarrow 0.
 \end{array}
\end{equation*}
In an analogous way,
\begin{equation*}
\begin{array}{c}
 \Vert\mathcal{D}_{2,n}(y', z', \widetilde{h}')\Vert^{2}_{\mathcal{U}}
\leq  C\left(\Vert z'(y_n-y)\Vert^{2}_{\mathcal{U}} + \Vert z'(z_n - z)\Vert^{2}_{\mathcal{U}} + \Vert y'(z_n - z)\Vert^{2}_{\mathcal{U}}\right) \\

\leq  C \Vert (y_n,z_n,\widetilde{h}_n) - (y,z,\widetilde{h})\Vert^{2}_{\mathcal{E}} \Vert (y',z',\widetilde{h}')\Vert^{2}_{\mathcal{E}}\rightarrow 0.
 \end{array}
\end{equation*}
The analysis for \(\mathcal{D}_{3,n}\) and \(\mathcal{D}_{4,n}\) is straightforward, and thus \eqref{convergencia_da_derivada} holds.

To verify the surjectivity condition of $\mathcal{A}'(0,0,0)$, consider \( (F_0, F_1, y_0, z_0) \in \mathcal{Z} \). According to Proposition \ref{control_of_SL}, we know that there exists \( (y, z, \widetilde{h}) \) that satisfies equation \eqref{linear_1} and fulfills the estimates \eqref{regularity_for_y_and_z} and \eqref{regularidade_em_H1}. Therefore, as a result, $(y,z,\widetilde{h})\in \mathcal{E}$ and $$\mathcal{A}'(0,0,0)(y,z,\widetilde{h})=(F_{0},F_{1},y_0,z_0).$$ 

Based on the previous verifications we conclude that Theorem \ref{Liusternik} can be applied to our configuration \( \mathcal{A}:\mathcal{E}\rightarrow\mathcal{Z}\).

\section*{Acknowledgements}
J. C. Barreira was financed in part by the Coordenação de Aperfeiçoamento de Pessoal de Nível Superior - Brasil (CAPES) - Finance Code 001.

M. Sonego has been partially supported by the Conselho Nacional de Desenvolvimento
Científico e Tecnológico (CNPq), Grant/Award Number: 311893/2022-8; Fundaçãoo
de Amparo à Pesquisa do Estado de Minas Gerais (FAPEMIG), Grant/Award
Number: RED-00133-21 and CAPES/Humboldt Research Scholarship Program, Number: 88881.876233/2023-01.

E. Zuazua was funded by the European Research Council (ERC) under the European Union’s
Horizon 2030 research and innovation programme (Grant No. 101096251-CoDeFeL), the Alexander von Humboldt-Professorship program, the ModConFlex Marie Curie Action, HORIZONMSCA-2021-dN-01, the Transregio 154 Project of the DFG, grants PID2020-112617GB-C22 and
TED2021131390B-I00 of the AEI (Spain), AFOSR Proposal 24IOE027, and Madrid GovernmentUAM Agreement for the Excellence of the University Research Staff in the context of the V
PRICIT (Regional Programme of Research and Technological Innovation).

\end{document}